\documentclass[a4paper,12pt,oneside]{article}
\usepackage[utf8]{inputenc}
\usepackage[english]{babel}
\usepackage{amsmath}
\usepackage{amsthm}
\usepackage{amssymb}
\usepackage{mathtools}
\usepackage{tikz}
\usepackage{verbatim}
\usepackage{fancyvrb}
\usepackage{array}
\usepackage{adjustbox}
\usepackage{changepage}
\usepackage{subcaption}
\usepackage{breqn}
\usepackage{enumerate}
\usepackage{wrapfig}
\usepackage{bbm}
\usepackage{bm}
\usepackage{enumitem}
\usepackage{csquotes}
\usepackage{hyperref}
\usepackage{framed}
\usepackage{enumitem}
\theoremstyle{definition}
\newtheorem{defi}{Definition}[section]
\newtheorem{thm}[defi]{Theorem}

\newtheorem{ex}[defi]{Example}
\newtheorem{prop}[defi]{Proposition}

\newtheorem{cor}[defi]{Corollary}
\newtheorem{rem}[defi]{Remark}

\newcommand{\N}{\mathbb{N}}
\newcommand{\Z}{\mathbb{Z}}

\newcommand{\C}{\mathbb{C}}
\newcommand{\R}{\mathbb{R}}

\newcommand{\im}{\Rightarrow}

\newcommand{\om}{\omega}

\newcommand{\sz}[1]{\left| \vec{#1} \right|}
\newcommand{\spn}[1]{\text{span}\{#1\}}
\newcommand{\MT}[1]{\mathcal{M}#1}
\newcommand{\LT}[1]{\mathcal{L}#1}
\newcommand{\T}[1]{\mathcal{T}#1}

\title{Multiple orthogonal polynomial ensembles of derivative type}
\author{Thomas Wolfs\footnote{Department of Mathematics, KTH Royal Institute of Technology, Sweden. \\
E-mail adress: \texttt{wolfs[at]kth.se}.}} 
\date{}

\begin{document}
\maketitle

\begin{abstract}

We characterize the biorthogonal ensembles that are both a multiple orthogonal polynomial ensemble and a polynomial ensemble of derivative type (also called a Pólya ensemble). We focus on the notions of multiplicative and additive derivative type that typically appear in connection with products and sums of random matrices respectively. Essential in the characterization is the use of the Mellin and Laplace transform: we show that the derivative type structure, which is a priori analytic in nature, becomes algebraic after applying the appropriate transform. Afterwards, we use the characterization to show that the eigenvalue densities of products of JUE and LUE matrices are essentially the only multiple orthogonal polynomial ensembles of multiplicative derivative type. We also show that the eigenvalue densities of sums of dilated LUE and GUE matrices are examples of multiple orthogonal polynomial ensemble of additive derivative type, but provide other examples as well. Finally, we explain how these notions of derivative type can be used to provide a partial solution to an open problem related to orthogonality of the finite finite free multiplicative and additive convolution of polynomials from finite free probability. In particular, we obtain families of multiple orthogonal polynomials that (de)compose naturally using these convolutions.
\medbreak


\end{abstract}

% {\small\tableofcontents}

\section{Introduction}
Over the last decade, a program was set out to obtain explicit expressions for the density of the eigenvalues of products and sums of random matrices. One of the first steps was taken in \cite{AB:prod_Gin} as Akemann and Burda obtained an expression for the density of the eigenvalues of products of square Ginibre matrices. A (complex) Ginibre matrix is a random matrix whose entries are independent standard complex Gaussian random variables, see \cite{Gin}. Afterwards, in \cite{AKW:prod_sq_Gin}, Akemann, Kieburg and Wei also managed to describe the squared singular values of products of square, and later in \cite{AIK:prod_rec_Gin}, rectangular, Ginibre matrices. Since this problem turned out to be more natural, see also \cite{KZ:prod_gin}, at least for products of random matrices focus shifted more towards their squared singular values rather than their eigenvalues. In this direction, an important milestone was achieved in \cite{KS:gin_prod} as Kuijlaars and Stivigny showed that it is possible to describe the squared singular values of products of any random matrix for which the squared singular values follow a so-called polynomial ensemble, with Ginibre matrices. Later in \cite{KKS:trunc_prod}, it was shown by Kieburg, Kuijlaars and Stivigny that one can also handle products with truncated Haar distributed random unitary matrices instead of Ginibre matrices. We refer to \cite{K:transf_PE} for an introduction to polynomial ensembles. The definition of a polynomial ensemble is as follows.

\begin{defi} \label{def_PE_full_}

Let $v_1,\dots,v_n:\R\to\R$ be such that $x\mapsto x^{k-1} v_j(x)$ is in $L^1(\R)$ for all $j,k\in\{1,\dots,n\}$. The polynomial ensemble associated to $v_1,\dots,v_n$ is the probability density $\text{PE}[v_1,\dots,v_n]$ on $\R^n$ given by
\begin{align} \label{PE}
    \frac{1}{Z_n} \prod_{1\leq j<k\leq n} (x_k-x_j) \det[v_j(x_k)]_{j,k=1}^{n} \geq 0,\quad \vec{x}\in\R^n,
\end{align}
where $Z_n\in\R_{\neq 0}$ is the normalization constant. We will write $\text{EV}(X)\sim\text{PE}[v_1,\dots,v_n]$ to indicate that an $n\times n$ random matrix $X$ has joint eigenvalue density \eqref{PE}.
\end{defi}

Polynomial ensembles arise as a special class of biorthogonal ensembles, see \cite{B:BE,KS:gin_prod}, and are therefore determinantal point processes, see \cite{J:DPP} for an introduction. A large class of polynomial ensembles are the so-called multiple orthogonal polynomial ensembles. They were introduced in \cite{Kuijlaars:MOPE} (see also \cite{AK:MOPE}) and generalize the well-studied class of orthogonal polynomial ensembles ($r=1$), also known as Coulomb gasses with $\beta=2$, see \cite{D:OPE} for an introduction. The correlation kernel for a multiple orthogonal polynomial ensemble is completely determined by the multiple orthogonal polynomials associated to the underlying system of weights, see \cite{Daems-Kuijlaars:CD_MOP}. Often, it is convenient to identify a given ensemble as a multiple orthogonal polynomial ensemble, because then a Riemann-Hilbert problem becomes available, which can be used to study the asymptotics of the kernel, see, e.g., \cite{BK:ext_src,BK:ext_src_1,ABK:ext_src_2,BK:ext_src_3}, \cite{KMFW:Bessel,KMFW:Bessel_1} and \cite{LY:ins_II,LY:ins_II_asym}.
\medbreak

We will use the following definition of a multiple orthogonal polynomial ensemble. Note that, unlike in \cite{Kuijlaars:MOPE}, we assumed that the associated multi-index is on the step-line 
$$\mathcal{S}^r = \{\vec{n}\in\N^r \mid n_1\geq\dots\geq n_r\geq n_1-1 \}.$$ 
This is standard practice when working with multiple orthogonal polynomials. Doing so, there is no ambiguity in the condition $n=\sz{n}$ as there is a unique multi-index $\vec{n}\in\mathcal{S}^r$ for which this equality holds.

\begin{defi} \label{def_MOPE_}
    We call the polynomial ensemble associated with $v_1,\dots,v_n:\R\to\R$ a multiple orthogonal polynomial ensemble with respect to $w_1,\dots,w_r:\R\to\R$, if
    \begin{align*}
        \spn{v_j(x) \mid j=1,\dots,n} = \spn{x^{k-1} w_j(x) \mid k=1,\dots,n_j, j=1,\dots,r},
    \end{align*}
for (the unique) $\vec{n}\in\mathcal{S}^r$ with $n=\sz{n}$. In that case, we write 
$$\text{PE}[v_1,\dots,v_n] = \text{MOPE}[n;w_1,\dots,w_r].$$
If $r=1$, we opt to write $\text{PE}[v_1,\dots,v_n] = \text{OPE}[n;w_1]$.
\end{defi}

For our convenience, given $r$ functions $w_1,\dots,w_r:X\to\R$ with $X\subset\R$, we will use the notation
    $$\mathcal{W}(n;w_1,\dots,w_r)=\spn{X\to\R:x\mapsto x^{k-1} w_j(x) \mid k=1,\dots,n_j, j=1,\dots,r},$$
where $\vec{n}\in\mathcal{S}^r$ is such that $\sz{n}=n$. Observe that there is some freedom in how the underlying functions are chosen as $\mathcal{W}(n;w_{0,1},\dots,w_{0,r})=\mathcal{W}(n;w_1,\dots,w_r)$ whenever $(w_{0,j})_{j=1}^n = (w_j)_{j=1}^n U$ and $U$ is an invertible upper-triangular $r\times r$ matrix. The functions $(w_{0,j})_{j=1}^n$ and $(w_j)_{j=1}^n U$ therefore give rise to the same multiple orthogonal polynomial ensemble.
\medbreak

The joint eigenvalue densities of Jacobi Unitary Ensemble (JUE), Laguerre Unitary Ensemble (LUE) and Gaussian Unitary Ensemble (GUE) are well-known examples of orthogonal polynomial ensembles and their correlation kernels involve the classical Jacobi, Laguerre and Hermite orthogonal polynomials respectively. Below, we recall the structure of these densities (which we can also see as the definition of the corresponding matrix ensembles).

\begin{ex}
Let $b>a>-1$ and let $X$ be an $n\times n$ JUE matrix with parameters $(a,b)$. We will also denote this as $X\sim\text{JUE}[n;a,b]$. Then $\text{EV}(X)\sim\text{OPE}[n;\mathcal{B}^{a,b}]$ in terms of the beta density $\mathcal{B}^{a,b}(x) = x^{a}(1-x)^{b-a-1}$ on $(0,1)$.
\end{ex}

It is well-known that the squared singular values of a truncated Haar distributed random unitary matrix follow a JUE, see, e.g., \cite[\S 3.1]{KS:gin_prod}.

\begin{ex}
Let $a>-1$ and let $X$ be an $n\times n$ LUE matrix with parameter $a$. We will also denote this as $X\sim\text{LUE}[n;a]$. Then $\text{EV}(X)\sim\text{OPE}[n;\mathcal{G}^a]$ in terms of the gamma density $ \mathcal{G}^a(x) = x^a e^{-x}$ on $\R_{>0}$.
\end{ex}

It is well-known that the squared singular values of a Ginibre matrix follow a LUE, see, e.g., \cite[\S 3.3]{KS:gin_prod}.

\begin{ex}
Let $X$ be an $n\times n$ GUE matrix. We will also denote this as $X\sim\text{GUE}[n]$. Then $\text{EV}(X)\sim\text{OPE}[n;\mathcal{G}]$ in terms of the Gaussian $ \mathcal{G}(x) = e^{-x^2}$ on $\R$.
\end{ex}

The next milestone was achieved in \cite{KK:SV_EV} as Kieburg and Kösters showed that instead of considering products of a random matrix for which the squared singular values follow a polynomial ensemble, with Ginibre and truncated random unitary matrices, one can actually handle products with any random matrices for which the squared singular values follow a so-called polynomial ensemble of derivative type. The JUE and LUE are examples of such ensembles, as explained in \cite[Ex. 2.4 $d)$ \& $c)$]{FKK:PolyaE}. In this multiplicative setting, we will define the notion of polynomial ensembles of derivative type through sets of functions of multiplicative derivative type, which we will introduce first. In the literature, typically only the associated polynomial ensemble is considered, but for our central result, the underlying functions don't necessarily have to give rise to a \textit{positive} polynomial ensemble. \\
Recall that a function $f:X\mapsto\R$ is locally absolutely continuous on $X\subset\R$, denoted by $f\in AC_{\text{loc}}(X)$, if and only if there exists $g\in L^1(X)$ and $\alpha\in\C$ such that $f(x) = \alpha + \int_1^x g(t) dt $ for all $x\in X$, and that in that case, $f$ is differentiable a.e. on $X$ with a.e. $f'=g$.

\begin{defi} \label{MDT_def}
    We call $n$ functions $v_1,\dots,v_n:\R_{>0}\to\R$ of multiplicative derivative type with respect to a function $\om\in C^{n-2}(\R_{>0})$, with $\om^{(n-2)} \in AC_{\text{loc}}(\R_{>0})$, if
    \begin{align*}
        \spn{v_j(x)}_{j=1}^n = \spn{ (x \frac{d}{dx})^{j-1}[\om(x)]}_{j=1}^n,\quad \text{a.e. } x\in\R_{>0}.
    \end{align*}
\end{defi}

Polynomial ensembles of multiplicative derivative type now arise from functions of multiplicative derivative type, whenever they give rise to a \textit{positive} polynomial ensemble. 

\begin{defi} \label{PE_MDT_def}
     We call the polynomial ensemble associated with $v_1,\dots,v_n:\R_{>0}\to\R$ of multiplicative derivative type with respect to $\om:\R_{>0}\to\R$, if the underlying functions $v_1,\dots,v_n$ have this property. In that case, we write 
        $$\text{PE}[v_1,\dots,v_n] = \text{PE}_{\text{MDT}}[n;\om].$$
\end{defi}

Unfortunately, in the literature on polynomial ensembles of multiplicative derivative type, there is some ambiguity about which space $\om^{(n-2)}$ belongs to. Originally, in \cite{KK:SV_EV}, Kieburg and Kösters demanded that $\om_n^{(n-2)}\in AC(\R_{>0})$. This is equivalent to the above because implicitly through Definition \ref{def_PE_full_}, we assume that $\om_n^{(n-1)}\in L^{1}(\R_{>0})$. Later in \cite{FKK:PolyaE}, Förster, Kieburg and Kösters demanded the stronger condition that $\om_n^{(n-2)}$ is differentiable on $\R$. For our purposes, it is important that weights of multiplicative derivative type are defined in the weaker way above as we want to focus solely on their structural properties and don't want to involve any additional properties coming from the associated polynomial ensemble.
\medbreak

The next milestone was achieved in \cite{KR:add_DT} as Kuijlaars and Román provided the analogue for the eigenvalues of sums of random matrices. This work was preceded by \cite{CKW:sums_prod}, where Claeys, Kuijlaars and Wang described the eigenvalues of the sum of a random matrix for which the eigenvalues follow a polynomial ensemble and a random matrix for which the eigenvalues follow a GUE. Generally, in the additive setting, it is possible to describe the eigenvalues of sums of a random matrix for which the eigenvalues follow a polynomial ensemble, with random matrices for which the eigenvalues follow a polynomial ensemble of another kind of derivative type. As explained in \cite[Ex. 2.4 $c)$ \& $b)$]{FKK:PolyaE}, the LUE and GUE are examples such ensembles. The precise definition of this additive derivative type is stated below. Similarly as before, our approach will be to define this notion through functions of additive derivative type, which we will introduce first. 

\begin{defi}\label{ADT_def}
      We call $n$ functions $v_1,\dots,v_n:\R\to\R$ of additive derivative type with respect to a function $\om\in C^{n-2}(\R)$, with $\om^{(n-2)}\in AC_{\text{loc}}(\R)$, if    
      \begin{align*}
        \spn{v_j(x)}_{j=1}^n = \spn{\om^{(j-1)}(x)}_{j=1}^n, \quad \text{a.e. } x\in\R.
    \end{align*}
\end{defi}

Polynomial ensembles of additive derivative type now arise from functions of additive derivative type, whenever they give rise to a \textit{positive} polynomial ensemble. 

\begin{defi}\label{PE_ADT_def}
      We call the polynomial ensemble associated with $v_1,\dots,v_n:\R_{>0}\to\R$ of additive derivative type with respect to $\om:\R\to\R$, if the underlying functions $v_1,\dots,v_n$ have this property. In that case, we write 
        $$\text{PE}[v_1,\dots,v_n] = \text{PE}_{\text{ADT}}[n;\om].$$
\end{defi}

Again, in the literature, there is some minor ambiguity about which space $\om^{(n-2)}$ belongs to. Originally, in \cite{KR:add_DT}, no explicit conditions on $\om$ are mentioned. Later in \cite{FKK:PolyaE}, Förster, Kieburg and Kösters demanded that $\om_n^{(n-2)}$ is differentiable on $\R$. For our purposes and because of the similarities to the multiplicative setting, we prefer the conditions stated above, which are less strong.
\medbreak

By now, polynomial ensembles of multiplicative and additive derivative type are relatively well-understood: in \cite{FKK:PolyaE}, the functions $\om$ in Definition~\ref{MDT_def}~\&~\ref{ADT_def} that give rise to a \textit{positive} polynomial ensemble of derivative type were characterized in terms of Pólya frequency functions. We refer to \cite[Ch. 7]{Karlin} for an introduction to Pólya frequency functions. As a consequence, such polynomial ensembles of derivative type are sometimes called Pólya ensembles. \\
Recently, several other notions of derivative type, compatible with other symmetry classes of random matrices, have been introduced as well. For example in \cite{FKK:PolyaE}, notions for random complex rectangular, Hermitian anti-symmetric and Hermitian anti-self-dual matrices were defined. In \cite{KLZF:cyclic_Polya}, a notion of derivative type compatible with random unitary matrices was considered.

\section{Main results}

Initially, the main goal of this work was to provide many examples of multiple orthogonal polynomial ensembles of derivative type. This was motivated by the fact that the associated multiple orthogonal polynomials would have desirable (de)composition properties with respect to the finite free multiplicative and additive convolution from finite free probability, see \cite{MSS:fin_free_conv}. Such (de)composition properties have recently received increased attention from the orthogonal polynomial community, because they are helpful in studying the properties of their zeros as explained in \cite{MFMP1,MFMP2}. A deeper connection between orthogonality and (de)compositions using the finite free multiplicative and additive convolution is suggested by an open problem posed by Martínez-Finkelshtein at the Hypergeometric and Orthogonal Polynomials Event in Nijmegen (May, 2024), see~\cite{HOPE}. Remarkably, after making the appropriate identifications, the open problem turns out to be equivalent to the one that drove the previously mentioned research around products and sums of random matrices. We will make this connection precise in Section \ref{FFP}. Doing so, we can formulate a partial solution to the open problem.
\medbreak

Ultimately, our study of multiple orthogonal polynomial ensembles of derivative type did not only lead to many new examples, but to a characterization of them. It turns out that the eigenvalue densities of products of JUE and LUE matrices are essentially the only multiple orthogonal polynomial ensembles of multiplicative derivative type and that the eigenvalue densities of sums of dilated LUE and GUE matrices are examples of multiple orthogonal polynomial ensemble of additive derivative type, but not the only ones. By identifying such ensembles as multiple orthogonal polynomial ensembles, a Riemann-Hilbert problem becomes available to perform an asymptotic analysis of the associated correlation kernel. Previously, such a connection was only known for the squared singular value densities of products of Ginibre matrices, see \cite{KZ:prod_gin}. Such densities arise as particular cases of the eigenvalue densities of products of LUE matrices.
\medbreak

Essential in the characterization, is the use of an appropriate integral transform. In the multiplicative setting, this role is played by the Mellin transform. In the additive setting, the Laplace transform is used. These transforms have several useful properties, which we will discuss in Section~\ref{PRELIM}. The idea of using such integral transforms is that the derivative type structure, which is a priori analytic in nature, becomes algebraic after applying the appropriate transform. The latter will be more prone to analysis and will reveal some implicit structure on the functions associated to a polynomial ensemble of derivative type. Results of this kind will be proven in Section \ref{DT}. In the particular case of a multiple orthogonal polynomial ensemble of derivative type, the algebraic structure will enable a characterization of the underlying weights.
\medbreak

A class of weights that give rise to multiple orthogonal polynomial ensembles of multiplicative derivative type is described in the result below. We will prove this in Section~\ref{CLASS_M}. The appropriate space $L^1_{\MT{},\Sigma}(\R_{>0})$ to which these weights should belong will be defined in Section~\ref{PRELIM_M}.

\begin{prop} \label{MOPE_MDT_w_om}
Let $n\in\N$ and take $\vec{n}\in\mathcal{S}^r$ with $\sz{n}=n$. Suppose that the Mellin transforms of $w_{1,1},\dots,w_{1,r}\in L^{1}_{\MT{},\Sigma}(\R_{>0})$ are given by
\begin{equation} \label{MOPE_MDT_w}
        \MT{w_{1,j}}(s) = c^s \, \prod_{i=1}^r \frac{\Gamma(s+a_i)^{d_1(i)}}{\Gamma(s+b_i)^{d_2(i)}} \frac{s^{j-1}}{\prod_{i=1}^j (s+b_i)^{d_2(i)}}, \quad s\in\Sigma,
\end{equation}
for all $j\in\{1,\dots,r\}$, and the Mellin transform of $\om_{1,n}\in L^{1}_{\MT{},\Sigma}(\R_{>0})$ is given by
\begin{equation} \label{MOPE_MDT_om}
    \MT{\om}_{1,n}(s) = c^s \prod_{i=1}^{r} \frac{\Gamma(s+a_i)^{d_1(i)}}{\Gamma(s+b_i+n_i)^{d_2(i)}},\quad s\in\Sigma,
\end{equation}
for some $a_i\in\C$, $b_i,c\in\R$, $d_1,d_2:\{1,\dots,r\}\to\{0,1\}$ with all $d_1(i)=1$ or all $d_2(i)=1$. Then $\text{MOPE}[n;w_{1,1},\dots,w_{1,r}]=\text{PE}_{\text{MDT}}[n;\om_{1,n}]$.
\end{prop}
\begin{rem} \label{MDT_MR_ALT}
    An alternative way to write down \eqref{MOPE_MDT_w} is to relabel the parameters such that we use a vector $\vec{a}\in\C^p$ on top and a vector $\vec{b}\in\R^q$ on the bottom with $\max\{p,q\}=r$. In order to describe the second factor, we then have to make use of an increasing and surjective function $\sigma:\{1,\dots,r\}\to\{1,\dots,q\}$. Doing so, we can write    
        $$\MT{w_{1,j}}(s) = c^s \,  \frac{\prod_{i=1}^p \Gamma(s+a_i)}{\prod_{i=1}^q \Gamma(s+b_i)} \frac{s^{j-1}}{\prod_{i=1}^{\sigma(j)} (s+b_i)}, \quad s\in\Sigma.$$
    Although this seems more insightful, it would be more involved to work with in the proofs.
\end{rem}

Most of the systems of weights described in \eqref{MOPE_MDT_w} have already been studied in \cite{L:MOP_HG,Wolfs}. In these works, the authors assume that
$$ d_1(i) = \begin{cases}
    1,\quad i\in\{1,\dots,p\}, \\
    0, \quad i\in\{p+1,\dots,r\},
\end{cases}\quad 
d_2(i) = \begin{cases}
    1,\quad i\in\{1,\dots,q\}, \\
    0, \quad i\in\{q+1,\dots,r\},
\end{cases} $$
with $\max\{p,q\}=r$, which would lead to weights with moments of the form
    $$\MT{w_{1,j}}(s) = c^s \, \frac{\prod_{i=1}^p \Gamma(s+a_i)}{\prod_{i=1}^q \Gamma(s+b_i)} \frac{s^{j-1}}{\prod_{i=1}^{\max\{j,q\}} (s+b_i)},\quad s\in \Z_{\geq 1}.$$
As explained in Remark \ref{MDT_MR_ALT}, the weights in \eqref{MOPE_MDT_w} arise in a similar way, but involve a slightly more general second factor  
$$\frac{s^{j-1}}{\prod_{i=1}^{\sigma(j)} (s+b_i)},$$
in terms of an increasing and surjective function $\sigma:\{1,\dots,r\}\to\{1,\dots,q\}$. \\
As explained in \cite{L:MOP_HG,Wolfs}, throughout the years, most of the systems of two weights covered by \eqref{MOPE_MDT_w} have already been studied in the context of multiple orthogonal polynomials, see \cite{VAYakubovich,KZ:prod_gin,SmetVA,LimaLoureiro1,LimaLoureiro2,VAWolfs}.
\medbreak

The characterization below shows that the ensembles associated to the above weights cover all the multiple orthogonal polynomial ensembles of multiplicative derivative type. We will prove the remaining equivalence between $i)$ and $ii)$ in Section \ref{CLASS_M}. Note that in $i)$, we only consider multiple orthogonal polynomial ensembles in which the underlying system of weights is independent of the number of particles that are modeled. The weight in the associated derivative type ensembles doesn't need to satisfy the same property (which is why we only indicate the class here and not the specific weight). 

\begin{thm} \label{MOPE_MDT}
    Let $v_1,\dots,v_n\in L^{1}_{\MT{},\Sigma}(\R_{>0})$ and $r<n$. Then the following are equivalent:
    \begin{itemize}
        \item[$i)$] there exists $w_1,\dots,w_r:\R_{>0}\to\R$ such that $\text{PE}[v_1,\dots,v_n]=\text{MOPE}[n;w_1,\dots,w_r]$ and $\text{MOPE}[m;w_1,\dots,w_r]$ is of multiplicative derivative type for all $m\in\{1,\dots,n\}$.
        \item [$ii)$] $\text{PE}[v_1,\dots,v_n]=\text{MOPE}[n;w_{1,1},\dots,w_{1,r}]$ with $w_{1,1},\dots,w_{1,r}$ as in \eqref{MOPE_MDT_w},
        \item [$iii)$] $\text{PE}[v_1,\dots,v_n]=\text{PE}_{\text{MDT}}[n;\om_{1,n}]$ with $\om_{1,n}$ as in \eqref{MOPE_MDT_om}.
    \end{itemize}
\end{thm}

Generically, the ensembles described in Theorem \ref{MOPE_MDT} $ii)$ and $iii)$ arise as the joint eigenvalue densities of products of JUE and LUE matrices. We will explain this in detail in Section \ref{EX_M}.
\medbreak

A class of weights that gives rise to multiple orthogonal polynomial ensembles of additive derivative type is described in the following result. We will prove this in Section \ref{CLASS_A}. The appropriate space $L^1_{\LT{},\Sigma}(\R)$ to which these weights should belong will be defined in Section \ref{PRELIM_A}.

\begin{prop} \label{MOPE_ADT_w_om}
Let $n\in\N$ and take $\vec{n}\in\mathcal{S}^r$ with $\sz{n}=n$. Suppose that the Laplace transforms of $w_{2,1},\dots,w_{2,r}\in L^{1}_{\LT{},\Sigma}(\R)$ are given by
\begin{equation} \label{MOPE_ADT_w}
        \LT{w_{2,j}}(s) = \exp\left( c  \int_{s_0}^{s} \prod_{i=1}^r \frac{(t+a_i)^{d_1(i)}}{(t+b_i)^{d_2(i)}} dt \right) \frac{s^{j-1}}{\prod_{i=1}^{j} (s+b_i)^{d_2(i)}}, \quad s\in\Sigma,
\end{equation}
for all $j\in\{1,\dots,r\}$, and the Laplace transform of $\om_{2,n}\in L^{1}_{\LT{},\Sigma}(\R)$ is given by
\begin{equation} \label{MOPE_ADT_om}
    \LT{\om}_{2,n}(s) =  \exp\left( c  \int_{s_0}^{s} \prod_{i=1}^r \frac{(t+a_i)^{d_1(i)}}{(t+b_i)^{d_2(i)}} dt \right) \frac{1}{\prod_{i=1}^{r} (s+b_i)^{d_2(i) n_i }},\quad s\in\Sigma,
\end{equation}
for some $a_i\in\C$, $b_i,c,s_0\in\R$, $d_1,d_2:\{1,\dots,r\}\to\{0,1\}$ with all $d_1(i)=1$ or all $d_2(i)=1$. Then $\text{MOPE}[n;w_{2,1},\dots,w_{2,r}]=\text{PE}_{\text{ADT}}[n;\om_{2,n}]$.
\end{prop}

\begin{rem}
    Similarly as in Remark \ref{MDT_MR_ALT}, there is an alternative way to write down \eqref{MOPE_ADT_w} in terms of parameters $\vec{a}\in\C^p$ and $\vec{b}\in\R^q$ with $\max\{p,q\}=r$. Indeed, after relabeling the parameters, we can write    
        $$\LT{w_{2,j}}(s) = \exp\left( c  \int_{s_0}^{s} \frac{\prod_{i=1}^p (t+a_i)}{\prod_{i=1}^q (t+b_i)} dt \right) \frac{s^{j-1}}{\prod_{i=1}^{\sigma(j)} (s+b_i)}, \quad s\in\Sigma,$$
    where $\sigma:\{1,\dots,r\}\to\{1,\dots,q\}$ is an increasing and surjective function.
\end{rem}

Generally, the systems of weights in \eqref{MOPE_ADT_w} have not been studied before, at least not in the context of multiple orthogonal polynomials. Only particular cases of systems of two weights have been studied in \cite{CVA:MOP_IB,LVA}. The former handles weights related to the $I$-Bessel functions, while one of the cases in the latter considers weights related to the Airy functions.
\medbreak

The characterization below shows that the ensembles associated to the above weights cover all the multiple orthogonal polynomial ensembles of additive derivative type. We will prove the remaining equivalence between $i)$ and $ii)$ in Section~\ref{CLASS_A}. Note that in $i)$, we again restrict to multiple orthogonal polynomial ensembles in which the underlying weights are independent of the number of particles that are modeled. The weight in the associated derivative type ensembles doesn't need to satisfy the same property (which is why we only indicate the class here and not the specific weight).

\begin{thm} \label{MOPE_ADT}
    Let $v_1,\dots,v_n\in L^{1}_{\LT{},\Sigma}(\R)$ and $r<n$. Then the following are equivalent:
    \begin{itemize}
        \item[$i)$] there exists $w_1,\dots,w_r:\R\to\R$ such that $\text{PE}[v_1,\dots,v_n]=\text{MOPE}[n;w_1,\dots,w_r]$ and $\text{MOPE}[m;w_1,\dots,w_r]$ is of additive derivative type for all $m\in\{1,\dots,n\}$.
        \item [$ii)$] $\text{PE}[v_1,\dots,v_n]=\text{MOPE}[n;w_{2,1},\dots,w_{2,r}]$ with $w_{2,1},\dots,w_{2,r}$ as in \eqref{MOPE_ADT_w},
        \item [$iii)$] $\text{PE}[v_1,\dots,v_n]=\text{PE}_{\text{MDT}}[n;\om_{2,n}]$ with $\om_{2,n}$ as in \eqref{MOPE_ADT_om}.
    \end{itemize}
\end{thm}

Several examples of the ensembles described in Theorem \ref{MOPE_ADT} $ii)$ and $iii)$ arise as the joint eigenvalue densities of sums of shifted LUE and GUE matrices, but generically, there are other examples as well. We will explain this in detail in Section \ref{EX_A}.
\medbreak

Observe that in Theorem \ref{MOPE_MDT} \& \ref{MOPE_ADT}, we didn't impose any explicit conditions on the parameters that can appear in \eqref{MOPE_MDT_w}, \eqref{MOPE_MDT_om}, \eqref{MOPE_ADT_w} \& \eqref{MOPE_ADT_om}. Results of this kind will be discussed in Section~\ref{EX}. Generally, we will only be able to provide sufficient conditions on the parameters that ensures that $\text{PE}[v_1,\dots,v_n]$ is well-defined, and in particular positive. Obtaining necessary conditions is more challenging because it is closely related to characterizing Pólya frequency functions of finite order, which has been an open problem since those of infinite order have been characterized in \cite{S:PFF}. For this reason, it would be interesting to obtain an optimal (or more optimal) set of conditions that guarantee this in the future. Such a result would then also lead to a strengthening of Theorem \ref{MOPE_MDT} \& \ref{MOPE_ADT}. For now, we do have the following complete characterization in the setting of orthogonal polynomial ensembles.
\medbreak

In the multiplicative setting, the only any point orthogonal polynomial ensembles of derivative type are the eigenvalue densities of JUE and LUE matrices. We will show this in Section \ref{EX_M}.

\begin{prop}\label{MDT_OPE}
Let $w\in L^{1}_{\MT{},\Sigma}(\R_{>0})$. Then, up to some dilation of $w$, the following are equivalent:
\begin{itemize}
        \item[$i)$] $\text{OPE}[n;w]$ is of multiplicative derivative type for all $n\in\N$,
        \item[$ii)$] $\text{OPE}[n;w]=\text{JUE}[n]$ for all $n\in\N$ or $\text{OPE}[n;w]=\text{LUE}[n]$ for all $n\in\N$.
    \end{itemize}
\end{prop}

In the additive setting, the only any point orthogonal polynomial ensembles of derivative type are the eigenvalue densities of LUE and GUE matrices. We will show this in Section \ref{EX_A}.

\begin{prop}\label{ADT_OPE}
Let $w\in L^{1}_{\LT{},\Sigma}(\R)$. Then, up to some dilation and shift of $w$, the following are equivalent:
\begin{itemize}
        \item[$i)$] $\text{OPE}[n;w]$ is of additive derivative type for all $n\in\N$,
        \item[$ii)$] $\text{OPE}[n;w]=\text{LUE}[n]$ for all $n\in\N$ or $\text{OPE}[n;w]=\text{GUE}[n]$ for all $n\in\N$.
    \end{itemize}
\end{prop}

\section{Preliminaries} \label{PRELIM}

In this section, we will review some properties of the Mellin and Laplace transform. These transforms will play an essential role in the study of functions of multiplicative and additive derivative type. Our main reference for the results involving the Mellin transform is \cite{BJ:MT}, which serves as an excellent introduction to the topic. Under the appropriate conditions, many of the results discussed there also carry over to the Laplace transform.

\subsection{Mellin transform} \label{PRELIM_M}

The Mellin transform of a function $f:\R_{>0}\to\C$ is given by
    $$(\MT{f})(s) = \int_{0}^\infty f(x) x^{s-1} dx, $$  
whenever this integral exists. For convenience, we will work with functions in the space
    $$ L_{\MT{},\Sigma}^{1}(\R_{>0}) = \{ f:\R_{>0}\to\C \mid \int_0^\infty |f(x)x^{s-1}| dx < \infty ,\ s\in\Sigma  \}, $$
where $\Sigma\subset\R_{>0}$ is an interval (or, if explicitly mentioned, a singleton). Clearly, instead of considering the interval $\Sigma$, we may also consider the strip $\Sigma+i\R$ in the right half of the complex plane. We may rephrase our results for intervals $-\Sigma$ in the left half of the complex plane by making use of the fact that $f\in L_{\MT{},\Sigma}^{1}(\R_{>0}) $ if and only if $\tilde{f}\in L_{\MT{},-\Sigma}^{1}(\R_{>0})$ with $\tilde{f}(x)=f(1/x)$.
\medbreak

The elementary examples below will play a role in the characterization of sets of multiplicative derivative type.

\begin{ex}[Eq. 5.2.1, \cite{DLMF}] \label{MDT_BETA}
    Let $b>a>-1$. The beta density $\mathcal{B}^{a,b}(x) = x^{a}(1-x)^{b-a-1}$ on $(0,1)$ has Mellin transform
        $$ \MT{\mathcal{B}^{a,b}}(s) = \Gamma(b-a) \frac{\Gamma(s+a)}{\Gamma(s+b)},\quad \text{Re}(s)>0. $$
\end{ex}

\begin{ex}[Eq. 5.12.1, \cite{DLMF}] \label{MDT_GAMMA}
    Let $a>-1$. The gamma density $ \mathcal{G}^a(x) = x^a e^{-x}$ on $\R_{>0}$ has Mellin transform
        $$ \MT{\mathcal{G}^a}(s) = \Gamma(s+a),\quad \text{Re}(s)>0.$$
\end{ex}

The following properties of the Mellin transform are well-known and will be used extensively throughout this work (see \cite[Thm. 1, 3 \& 7]{BJ:MT}). First, we recall the fact that the Mellin transform of a function $f\in L_{\MT{},\Sigma}^{1}(\R_{>0})$ is always analytic on the associated strip $\Sigma+i\R$. Second, whenever $f,g\in L_{\MT{},\{c\}}^{1}(\R_{>0})$, there exists a Mellin convolution
    $$(f\ast_{\mathcal{M}} g)(x) = \int_{0}^\infty f(t) g\left(\frac{x}{t}\right) \frac{dt}{t},\quad \text{a.e. } x\in\R_{>0},$$
with the property that
    $$ (\MT{[f\ast_{\mathcal{M}} g]})(s) = (\MT{f})(s) \cdot (\MT{g})(s),\quad s\in c+i\R. $$
Lastly, whenever $f\in L_{\MT{},\{c\}}^{1}(\R_{>0})$ and $\MT{f}\in L^1(c+i\R)$, there is an inverse transform
    $$ f(x) = \int_{c-i\infty}^{c+i\infty} \MT{f}(s) x^{-s} \frac{ds}{2\pi i},\quad \text{a.e. } x\in\R_{>0}.  $$
In particular, this implies that the Mellin transform is injective on $L_{\MT{},\{c\}}^{1}(\R_{>0})$ in an a.e.-sense.
\medbreak

The results below are less standard, but will play an important role later. The first result, known as the Riemann-Lebesgue lemma for the Mellin transform, will be used to establish whether certain functions can appear as a Mellin transform.

\begin{prop}[Thm. 2, \cite{BJ:MT}] \label{MT_RL}
    Suppose that $f\in L_{\MT{},\Sigma}^{1}(\R_{>0})$. Then, for every $s\in\Sigma$, we have
        $$ \lim_{|t|\to\infty} \MT{f}(s+it) = 0. $$
\end{prop}

The second result describes the behavior of the Mellin transform with respect to the differential operator that appears in Definition \ref{MDT_def}. Typically the conditions on the underlying function are stronger and one can use integration by parts, but for us it is crucial that the conditions are minimal and that we have the stated equivalence.

\begin{prop}[Cor. 7, \cite{BJ:MT}] \label{MT_DIFF}
    Suppose that $f\in L_{\MT{},\Sigma}^{1}(\R_{>0})$, let $r\in\N$ and consider the differential operator $(Dg)(x)=-xg'(x)$. Then the following are equivalent:
    \begin{itemize}
        \item[i)] there exists $F\in C^{r-1}(\R_{>0})$, with $F^{(r-1)}\in AC_{\text{loc}}(\R_{>0})$ and $D^r F\in L_{\MT{},\Sigma}^{1}(\R_{>0})$, such that $f=F$ a.e. on $\R_{>0}$,
        \item[ii)] there exists $g\in L_{\MT{},\Sigma}^{1}(\R_{>0})$ such that $\MT{g}(s) = s^r \MT{f}(s)$ for all $s\in\Sigma+i\R$.
    \end{itemize}
    In that case, $g=D^r F$ a.e. on $\R_{>0}$.
\end{prop}

The last result handles uniqueness of solutions of first order difference equations that arise as Mellin transforms of positive functions. It essentially follows from the generalization of the Bohr-Mollerup theorem in \cite{MZ:B-M} (see also \cite[Thm. 3.1]{W:B-M} for another formulation with slightly stronger conditions).

\begin{prop} \label{MT_DE_uni}
    Suppose that $g:\R_{>s_0}\to\R_{>0}$ satisfies $\lim_{n\to\infty} g(n+1)/g(n)=1$. Then, up to a scalar multiplication, the difference equation $F(s+1)=g(s)F(s)$ for $s>s_0$ has at most one solution of the form $\MT{f}$ with $f\in L_{\MT{},\R_{>s_0}}^{1}(\R_{>0})$ and $f>0$ a.e. on $\R_{>0}$.
\end{prop}
\begin{proof}
    We will first show the result in the special case that $s_0=0$. In fact, we will show something stronger, namely that, up to a scalar multiplication, there is at most one strictly positive log-convex solution. The fact that Mellin transform $\MT{f}$ of a function $f\in L_{\MT{},\R_{>0}}^{1}(\R_{>0})$ with $f>0$ a.e. is log-convex follows in a straightforward way from Hölder's inequality. 
    \begin{comment}
    $$\begin{aligned}
        \ln\MT{f}(\lambda s_1+(1-\lambda)s_2) &= \ln \int_0^\infty f(x) x^{\lambda s_1+(1-\lambda)s_2-1} dx \\
        &= \ln \int_0^\infty (f(x) x^{s_1-1})^{\lambda} (f(x) x^{s_2-1})^{1-\lambda} dx \\
        &\leq \ln[ (\int_0^\infty f(x) x^{s_1-1} dx)^{\lambda} (\int_0^\infty f(x) x^{s_2-1} dx)^{1-\lambda} ] \\
        &= \lambda \ln\MT{f}(s_1) + (1-\lambda) \ln\MT{f}(s_2).
    \end{aligned}$$
    \end{comment}
    Note that by taking logarithms, any strictly positive log-convex solution of $F(s+1)=g(s)F(s)$ for $s>0$, corresponds to a real-valued convex solution of $(\Delta F)(s) = \ln g(s) $ for $s>0$ (here $\Delta$ denotes the forward difference operator $(\Delta F)(s)=F(s+1)-F(s)$). However, as $\lim_{n\to\infty} (\Delta \ln g)(n) = 0 $, the latter are uniquely determined by $\ln g$ up to an additive constant due to \cite[Thm. 1.4]{MZ:B-M}.
   
    In order to extend to general $s_0\geq 0$, we note that any solution of $F(s+1)=g(s)F(s)$ for $s>s_0$ of the form $\MT{f}$ gives rise to a solution of $F(s+1)=g(s)F(s)$ for $s>0$ of the form $\MT{f_0}$ through the operator $f_0(x) = x^{s_0} f(x)$. Clearly still $f_0>0$ a.e. and since $\MT{f_0}(s) = \MT{f}(s+s_0)$ for $s\in\R_{>0}$, we also have $f_0\in L_{\MT{},\R_{>0}}^{1}(\R_{>0})$.
\end{proof}

\subsection{Laplace transform} \label{PRELIM_A}

The Laplace transform can be seen as the additive analogue of the Mellin transform. The (bilateral) Laplace transform of a function $f:\R\to\C$ is given by 
    $$(\mathcal{L} f)(s) = \int_{-\infty}^\infty f(x) e^{-sx} dx, $$
whenever this integral exists. For convenience, we will work with functions in the space
$$ L_{\LT{},\Sigma}^{1}(\R) = \{ f:\R\to\C \mid \int_{-\infty}^{\infty} |f(x)e^{-sx}| dx < \infty,\ s\in\Sigma  \}, $$
where $\Sigma\subset\R_{>0}$ is an interval (or, if explicitly mentioned, a singleton). Note that instead of considering the interval $\Sigma$, we could also have considered the strip $\Sigma+i\R$ in the right half of the complex plane. It is straightforward to see that, given $f:\R\to\C$, by defining $\tilde{f}: \R_{>0}\to\C:x\mapsto f(-\ln x)$, we have $f\in L_{\LT{},\Sigma}^{1}(\R)$ if and only if $\tilde{f}\in L_{\MT{},\Sigma}^{1}(\R)$. In that case, $\LT{f}=\MT{\tilde{f}}$ on $\Sigma$. As a consequence, many results for the Mellin transform can be rephrased in terms of the Laplace transform.
\medbreak

The following elementary examples will play a role in the characterization of sets of additive derivative type.

\begin{ex}[Eq. 5.9.1, \cite{DLMF}] \label{ADT_GAMMA}
Let $a>-1$. The gamma density $ \mathcal{G}^a(x) = x^a e^{-x}$ on $\R_{>0}$ has Laplace transform
    $$\LT{\mathcal{G}^a}(s) = \frac{\Gamma(a+1)}{(s+1)^{a+1}},\quad \text{Re}(s)>0.$$
\end{ex}

\begin{ex} [Table 1.14.1, \cite{DLMF}] \label{ADT_GAUSSIAN}
The Gaussian density $\mathcal{G}(x) = e^{-x^2}$ on $\R$ has Laplace transform
    $$\LT{\mathcal{G}}(s) = \sqrt{\pi} e^{s^2/4},\quad \text{Re}(s)>0.$$
\end{ex}

In what follows, we will state some well-known properties of the Laplace transform. These properties essentially carry over from the corresponding results for the Mellin transform. For example, the Laplace transform of a function $f\in L_{\LT{},\Sigma}^{1}(\R)$ is also analytic on the associated strip $\Sigma+i\R$. The convolution law and inversion theorem take the following form. Whenever $f,g\in L_{\LT{},\{c\}}^{1}(\R_{>0})$, there exists a Laplace convolution
    $$(f\ast_{\mathcal{L}} g)(x) = \int_{0}^\infty f(t) g(x-t) dt,\quad \text{a.e. } x\in\R,$$
with the property that
    $$ (\LT{[f\ast_{\mathcal{L}} g]})(s) = (\LT{f})(s) \cdot (\LT{g})(s),\quad s\in c+i\R. $$
Whenever $f\in L_{\LT{},\{c\}}^{1}(\R_{>0})$ and $\LT{f}\in L^1(c+i\R)$, there is an inverse transform
    $$ f(x) = \int_{c-i\infty}^{c+i\infty} \LT{f}(s) e^{sx} \frac{ds}{2\pi i},\quad \text{a.e. } x\in\R.  $$
In particular, this implies that the Laplace transform is injective on $L_{\LT{},\{c\}}^{1}(\R)$ in an a.e.-sense.
\medbreak

In the setting of the Laplace (or rather Fourier) transform, the Riemann-Lebesgue lemma is also standard. Later, we will use this result to establish whether certain functions can appear as a Laplace transform.

\begin{prop} \label{LT_RL}
    Suppose that $f\in L_{\LT{},\Sigma}^{1}(\R)$. Then, for every $s\in\Sigma$, we have
        $$ \lim_{|t|\to\infty} \LT{f}(s+it) = 0. $$
\end{prop}

Less standard is the following analogue of Proposition \ref{MT_DIFF}. Since, to our knowledge such a result for the Laplace transform (or Fourier transform) with minimal conditions on the underlying functions hasn't appeared before in the literature, we provide an explicit self-contained proof (based on some ideas that appear in \cite[\S 8]{BJ:MT}). Alternatively, one could use make use of the corresponding result for the Mellin transform.

\begin{prop} \label{LT_DIFF}
    Suppose that $f\in L_{\LT{},\Sigma}^{1}(\R)$ and let $r\in\N$. Then the following are equivalent:
    \begin{itemize}
        \item[$i)$] there exists $F\in C^{r-1}(\R)$, with $F^{(r-1)}\in AC_{\text{loc}}(\R)$ and $F^{(r)}\in L_{\LT{},\Sigma}^{1}(\R)$, such that $f=F$ a.e. on $\R$,
        \item[$ii)$] there exists $g\in L_{\LT{},\Sigma}^{1}(\R)$ such that $\LT{g}(s) = s^r \LT{f}(s)$ for all $s\in\Sigma+i\R$.
    \end{itemize}
    In that case, $g=F^{(r)}$ a.e. on $\R$.
\end{prop}
\begin{prop} \label{LT_DIFF}
    Suppose that $f\in L_{\LT{},\Sigma}^{1}(\R)$ and let $r\in\N$. Then the following are equivalent:
    \begin{itemize}
        \item[$i)$] there exists $F\in C^{r-1}(\R)$, with $F^{(r-1)}\in AC_{\text{loc}}(\R)$ and $F^{(r)}\in L_{\LT{},\Sigma}^{1}(\R)$, such that $f=F$ a.e. on $\R$,
        \item[$ii)$] there exists $g\in L_{\LT{},\Sigma}^{1}(\R)$ such that $\LT{g}(s) = s^r \LT{f}(s)$ for all $s\in\Sigma+i\R$.
    \end{itemize}
    In that case, $g=F^{(r)}$ a.e. on $\R$.
\end{prop}
\begin{proof}
    We will first prove the implication $i)\im ii)$. We will show that $\LT{F^{(r)}}(s)=s^r\LT{f}(s)$ for all $s\in\Sigma$. The challenge is to circumvent the fact that the Laplace transform of the intermediate derivatives doesn't need to exist, otherwise we could use an argument based on integration by parts. We will compute the Laplace transform of
        $$ F_0(x) = \int_0^1 \dots \int_0^1 F^{(r)}(x+x_1+\dots+x_r) dx_r\dots dx_1,\quad x\in\R,  $$
    in two ways. The function $F_0:\R\to\C$ is well-defined because $F\in C^{r-1}(\R)$ and $F^{(r-1)}\in AC_{\text{loc}}(\R)$, hence $F^{(j)}\in L^1_{\text{loc}}(\R)$ for all $j\in\{0,\dots,r\}$. In fact, in that case, we can use the identity $(\Delta h)(x) = \int_0^1 h'(x+t) dt $ to show that $F_0 = \Delta^r F$ on $\R$ and thus (see \cite[Eq. 3.9.4]{DLMF})
        $$ F_0(x) = \sum_{k=0}^{r} \binom{r}{k} (-1)^{r-k} F(x+k),\quad x\in\R. $$
    Its Laplace transform is then given by
        $$ \LT{F_0}(s) = \sum_{k=0}^{r} \binom{r}{k} (-1)^{r-k} e^{sk} \LT{F}(s) = (e^s-1)^r \LT{F}(s),\quad s\in\Sigma. $$
    On the other hand, since $F^{(r)}\in L_{\LT{},\Sigma}^{1}(\R)$, we can also compute the Laplace transform of $F_0$ directly by applying Fubini's theorem and performing a change of variables
        $$ \LT{F_0}(s) = \int_0^1 \dots \int_0^1  e^{s(x_1+\dots+x_r)} \LT{F^{(r)}}(s) dx_r\dots dx_1 =  \frac{(e^s-1)^r}{s^r} \LT{F^{(r)}}(s),$$
    for $s\in\Sigma$. Combining both results, we obtain the desired identity.
    \medbreak
    
    We will now prove the other implication $ii)\im i)$. Since $g,1_{\R_{>0}}\in L_{\LT{},\Sigma}^{1}(\R)$, we can consider the Laplace convolutions $g_j = g \ast_{\LT{}} (\ast_{\LT{}} 1_{\R_{>0}})^{r-j}$ for $j\in\{0,\dots,r\}$, which are a priori defined a.e. on $\R$. Moreover, we have $g_j\in L_{\LT{},\Sigma}^{1}(\R)$ and by the assumption
        $$\LT{g_j}(s) = \frac{\LT{g}(s)}{s^{r-j}} = s^{j} \LT{f}(s),\quad s\in\Sigma. $$
    As a consequence, by taking appropriate linear combinations, for every $c\in\R$, there exists a function $g_c\in L_{\LT{},\Sigma}^{1}(\R)$ such that
        $$\LT{g_c}(s) = (s-c)^r \LT{f}(s),\quad s\in\Sigma.$$
    Now we fix $c\in\Sigma$ and consider
        $$ F_c(x) = e^{cx} \int_{-\infty}^{x} \int_{-\infty}^{x_1} \dots \int_{-\infty}^{x_{r-1}} g_c(x_r) e^{-cx_r} dx_r\dots dx_1,\quad x\in\R. $$
    Observe that $F_c = g_c \ast_{\LT{}} (\ast_{\LT{}} \exp_c )^r$, with $\exp_c(x)=\exp(cx) 1_{\R_{>0}}(x)$ on $\R$, and that $g_c,\exp_c\in L_{\LT{},\Sigma_c}^{1}(\R)$, with $\Sigma_c = \{s\in\Sigma \mid s>c\}$. The latter is non-empty, because $\Sigma$ is an open interval and $c\in\Sigma$. As a consequence, $F_c$ is well-defined a.e. on $\R$. Moreover, we have $F_c\in L_{\LT{},\Sigma_c}^{1}(\R)$ and 
        $$ \LT{F_c}(s) = \frac{\LT{g}_c(s)}{(s-c)^r} = \LT{f}(s),\quad s\in\Sigma_c. $$     
    Injectivity of the Laplace transform then implies that $F_c=f$ a.e. on $\R$. We will show that $F_c$ satisfies the other conditions in $i)$ as well. Denote, for every $j\in\{0,\dots,r-1\}$,
        $$ F_{j,c}(x) = e^{cx} \int_{-\infty}^{x} \int_{-\infty}^{x_{j+1}} \dots \int_{-\infty}^{x_{r-1}} g_c(x_r) e^{-cx_r} dx_r\dots dx_{j+1},\quad x\in\R. $$
    Then we can show by induction that, for every $n\in\{0,\dots,r-1\}$,
        $$ F_{c}^{(n)}(x) = \sum_{k=0}^{n} \binom{n}{k} c^{n-k} F_{k,c}(x),\quad x\in\R, $$
    and that $F_{c}^{(n)}(x)\in C^1(\R)$, for $n\in\{0,\dots,r-2\}$, and $F_{c}^{(r-1)}\in AC_{\text{loc}}(\R)$. To this end, we can use the fact that functions of the form $x\mapsto \int_{-\infty}^x h(t) dt$ are in $AC_{\text{loc}}(\R)$ for $h\in L^{1}(\R)$ and are in $C^{1}(\R)$ for $h\in C^{0}(\R)$. Since $F_{c}^{(r-1)}\in AC_{\text{loc}}(\R)$, its derivative exists a.e. and is given by
        $$ F_{c}^{(r)}(x) = \sum_{k=0}^{r} \binom{r}{k} c^{r-k} F_{k,c}(x),\quad \text{a.e. } x\in\R. $$
    It then remains to show that $F_c^{(r)}\in L_{\LT{},\Sigma}^{1}(\R)$. Observe that
        $$\LT{F_{k,c}}(s) = \frac{\LT{g_c}(s)}{(s-c)^{r-k}} = (s-c)^k \LT{f}(s),\quad s\in\Sigma_c,$$
    and therefore
        $$ \LT{F_c^{(r)}}(s) =  \sum_{k=0}^{r} \binom{r}{k} c^{r-k} (s-c)^k \LT{F_c}(s) = s^{r} \LT{f}(s),\quad s\in\Sigma_c. $$
    Therefore, by the assumption, $\LT{F_c^{(r)}}(s) = \LT{g}(s)$ for all $s\in\Sigma_c$. Since the Laplace transform is injective, we can then conclude that $F_c^{(r)} = g$ a.e. and hence $F_c^{(r)}\in L_{\LT{},\Sigma}^{1}(\R)$.    
\end{proof}

Note that in several steps of the above proof, we used implicitly that $0\not\in\Sigma$, which is why we restricted to intervals $\Sigma\subset\R_{>0}$ at the start of this section.

\section{Examples} \label{EX}

In this section, we will provide many examples of the ensembles described in Theorem~\ref{MOPE_MDT}~\&~\ref{MOPE_ADT}. In the multiplicative setting, we will show that the eigenvalue densities of products of JUE and LUE matrices fit into this framework. In fact, they turn out to be the generic examples of ensembles covered by Theorem~\ref{MOPE_MDT}. In the additive setting, we will show that the eigenvalue densities of sums of (dilated) LUE and GUE matrices fit into the framework of Theorem~\ref{MOPE_ADT}. However, here it turns out that we can give many other examples as well. In both settings, we will analyze the values of the parameters under which the functions in \eqref{MOPE_MDT_om} and \eqref{MOPE_ADT_om} give rise to a \textit{positive} ensemble of derivative type, in which case Theorem~\ref{MOPE_MDT}~\&~\ref{MOPE_ADT} is applicable. This will allow us to show  Proposition~\ref{MDT_OPE}~\&~\ref{ADT_OPE}.

\subsection{Multiplicative setting} \label{EX_M}

The eigenvalue densities of JUE and LUE matrices are elementary examples of the (multiple) orthogonal polynomial ensembles covered by Theorem \ref{MOPE_MDT}. This follows immediately from Example \ref{MDT_BETA} \& \ref{MDT_GAMMA}. The fact that they are of multiplicative derivative type was first observed in \cite[Ex. 2.4]{FKK:PolyaE}.

\begin{cor} \label{OPE_MDT_EX}
Suppose that $b>a>-1$.
\begin{itemize}
    \item[$i)$] If $X\sim\text{GUE}[n;a]$, i.e. $\text{EV}(X)\sim\text{OPE}[n;\mathcal{G}^a]$, then $\text{EV}(X)\sim\text{PE}_{\text{MDT}}[n;\mathcal{G}^a]$.
    \item[$ii)$] If $X\sim\text{JUE}[n;a,b]$, i.e. $\text{EV}(X)\sim\text{OPE}[n;\mathcal{B}^{a,b}]$, then $\text{EV}(X)\sim\text{PE}_{\text{MDT}}[n;\mathcal{B}^{a,b+n-1}]$.
\end{itemize}
\end{cor}

Since the multiplicative derivative type structure of the eigenvalue densities of independent random matrices is preserved under taking their products, i.e. if $X_1$ and $X_2$ are independent random matrices with $\text{EV}(X_1)\sim\text{PE}_{\text{MDT}}[n;\om_1]$ and $\text{EV}(X_2)\sim\text{PE}_{\text{MDT}}[n;\om_2]$, then $\text{EV}(X_1X_2)\sim\text{PE}_{\text{MDT}}[n;\om_1\ast_{\MT{}}\om_2]$, see \cite{FKK:PolyaE}, the above leads to many other examples of polynomial ensembles of derivative type. Theorem \ref{MOPE_MDT} then shows that they are also multiple orthogonal polynomial ensembles.

\begin{cor} \label{MOPE_MDT_EX}
     Let $n\in\N$ and take $\vec{n}\in\mathcal{S}^r$ with $\sz{n}=n$. Suppose that all $a_i>-1$ and that all $b_i+n_i-n+1>a_i$. Let $ X_i\sim\text{JUE}[n;a_i,b_i+n_i-n+1]$, for $j\in\{1,\dots,q\}$, and $ X_i\sim\text{GUE}[n;a_i]$, for $i\in\{q+1,\dots,r\}$, be independent random matrices. Then $\text{EV}(X_1\dots X_r)\sim\text{PE}_{\text{MDT}}[n;\om_{n,1}]$
    with 
        $$ \MT{\om_{1,n}}(s) =  \frac{\prod_{i=1}^p \Gamma(s+a_i)}{\prod_{i=1}^q \Gamma(s+b_i+n_i)},\quad \text{Re}(s)>0, $$
    and $\text{EV}(X_1\dots X_r)\sim\text{MOPE}[n;w_{1,1},\dots,w_{1,r}]$ with
        $$ \MT{w_{1,j}}(s) = \frac{\prod_{i=1}^p \Gamma(s+a_i)}{\prod_{i=1}^q \Gamma(s+b_i)} \frac{s^{j-1}}{\prod_{i=1}^{\max\{j,q\}} (s+b_i)},\quad \text{Re}(s)>0.$$
\end{cor}

It turns out that we can slightly weaken the conditions on the parameters and still have that $\om_{1,n}$ gives rise to a \textit{positive} ensemble of derivative type, but then we lose the connection to JUE and LUE matrices. The case where the parameters are integers is essentially covered by the techniques from \cite[Cor. 2.4]{KKS:trunc_prod} and \cite[Thm. 2.1]{KS:gin_prod}. In \cite[Thm. 2.19]{CKW:sums_prod}, special care was given to the setting with all $d_1(i)=d_2(i)=1$ and an earlier question of optimally of the conditions in this setting was answered.
\medbreak

Even though Corollary \ref{MOPE_MDT_EX} doesn't cover all of the ensembles described by Theorem \ref{MOPE_MDT}, the result below, in particular $i)$, shows that they can be seen as the generic examples. As discussed in the introduction, it would be interesting to obtain an optimal set of conditions for which we the weights $\om_{1,n}$ in \eqref{MOPE_MDT_w} give rise to a \textit{positive} polynomial ensemble of multiplicative derivative type. A first step into this direction is the following in which we exploit the fact that the underlying space $L_{\MT{},\Sigma}^1(\R_{>0})$ induces certain conditions on the parameters.

\begin{prop} \label{MDT_COND}
In the setting of Theorem \ref{MOPE_MDT}, we have
\begin{itemize}
    \item[$i)$] all $d_1(i)=1$, 
    \item[$ii)$] all $\text{Re}\,a_i>-1$ and $\sum_{i=1}^{r} \text{Im}\,a_i=0$, 
    \item[$iii)$] $\sum_{i=1}^r (b_i-a_i) > - 1$ if all $d_2(i)=1$.
\end{itemize}
\end{prop}
\begin{proof}
    We will first show a weaker version of $ii)$, namely that all $d_1(i) \text{Re}\,a_i>-1$ and $\sum_{i=1}^{r} d_1(i) \text{Im}\,a_i=0$. After showing $i)$, this then gives $ii)$. For the first part, we note that $\Sigma$ contains $\Z_{\geq 1}$ (by integrability of the moments), hence it must be of the form $\Sigma=(s_0,\infty)$ for some $s_0\in[0,1)$. However, $\Gamma(s+a_i)^{d_1(i)}$ needs to be defined for all $s\in\Sigma$, so we must have that $d_1(i)\text{Re}(a_i)\geq -s_0>-1$. The second part follows from the fact that $\MT{v_1}$ is real-valued, hence $d_1(i) a_i\in\C\backslash\R$ must appear in conjugate pairs (consider, e.g., $\MT{v_1}(2)/\MT{v_1}(1)$).  \\    
    We will now use this, in combination with Proposition \ref{MT_RL}, to prove $i)$ and $iii)$. Since $w_{1,1}\in L^1_{\MT{},\Sigma}(\R_{>0})$, for every $c\in\Sigma$, we must have
        $$ \lim_{|t|\to\infty} \MT{w_{1,1}}(c+it) = 0. $$
    Denote $p=|d_1^{-1}(\{1\})|$ and $q=|d_2^{-1}(\{1\})|$. Stirling's asymptotic formula (see, e.g., \cite[Eq. 5.11.7]{DLMF}) implies that
        $$ \MT{w_{1,1}}(c+it) \asymp (it)^{(p-q)(c+it-\frac{1}{2})+\sum_{i=1}^r (d_1(i) a_i-d_2(i) b_i)-d_2(1)} e^{-(p-q)(c+it)} , $$
    as $|t|\to\infty$. Recall that the notation $f(t)\asymp g(t)$ as $t\to\infty$ denotes that $\lim_{t\to\infty} f(t)/g(t) = C$ for some $C\in\R_{\neq 0}$. If we write $t=|t| e^{i\arg t}$ and $i=e^{i\frac{\pi}{2}}$, the above becomes
    $$ |\MT{w_{1,1}}(c+it)|\, \asymp |t|^{(p-q)(c-\frac{1}{2})+\sum_{i=1}^r (d_1(i) \text{Re}(a_i)- d_2(i) b_i) -d_2(1)} e^{-(p-q)(c+t(\frac{\pi}{2}+\arg t))}, $$
    as $|t|\to\infty$. Note that here we used that $\sum_{i=1}^r d_1(i) \text{Im}\,a_i=0$. Assume now that some $d_1(i)=0$, i.e. $p<q$, then the power of the exponential factor is strictly positive as $t\to\infty$ and thus $\lim_{t\to\infty}|\MT{v_1}(c+it)|\, = \infty$. Hence we must have $i)$.\\
    Assume now that all $d_1(i)=1$ and all $d_2(i)=1$, i.e. $p=q$, then the power of the exponential factor is $0$ and we must have $\sum_{i=1}^r (d_1(i) \text{Re}\,a_i-b_i) - 1 < 0$. Together with $i)$ and $ii)$, this implies $iii)$.
\end{proof}

A consequence of the above is that one of the classes of weights that appears in Theorem~\ref{MOPE_MDT} necessarily have a sign change in $\R_{>0}$ and therefore don't give rise to a \textit{positive} ensemble.

\begin{cor}
    A function $v:\R_{>0}\to\R$ with a Mellin transform as in \eqref{MOPE_MDT_w} or \eqref{MOPE_MDT_om}, with some $d_1(i)=0$, possesses a sign change in $\R_{>0}$.
\end{cor}
\begin{proof}
    If $v$ doesn't have a sign change, we have $\MT{|v|}=\pm \MT{v}$ and therefore $v\in L^1_{\MT{},\Sigma}(\R_{>0})$. In that case, we can apply Proposition \ref{MDT_COND} to get a contradiction.
\end{proof}

By considering the simplest setting covered by Theorem \ref{MOPE_MDT} and Proposition \ref{MDT_COND}, i.e. $r=1$, we can show that the eigenvalue densities of JUE and LUE matrices are essentially the only orthogonal polynomial ensembles of multiplicative derivative type. This is the content of Proposition~\ref{MDT_OPE}, which we will prove now. 

\begin{proof}[Proof of Proposition \ref{MDT_OPE}]
    The fact that the eigenvalue densities of JUE and LUE matrices are orthogonal polynomial ensembles of multiplicative derivative type was already recorded in Corollary \ref{OPE_MDT_EX}. By Theorem \ref{MOPE_MDT}, the latter are characterized by a weight $w_{1,1}\in L^{1}_{\MT{},\Sigma}(\R_{>0})$ with a Mellin transform of the form
        $$\MT{w_{1,1}}(s) = c^s \, \frac{\Gamma(s+a)^{d_1}}{\Gamma(s+b)^{d_2}} \frac{1}{(s+b_2)^{d_2}},\quad s\in\Sigma,$$
    with $a\in\C$, $b,c\in\R$ and $d_1,d_2\in\{0,1\}$ with $\max\{d_1,d_2\}=1$. Since $c\neq 0$, without loss of generality, we can assume that $c=1$. Otherwise, we can consider the dilated weight $\tilde{w}_{1,1}(x)=w_{1,1}(cx)$ instead. By Proposition \ref{MDT_COND}, we must have $d_1=1$, $a>-1$ and $b-a>-1$ (so $b>-2$). Suppose that $d_2=0$, then 
    $$\MT{w_{1,1}}(s) = \Gamma(s+a),\quad s\in\Sigma,$$
    Hence, by Example \ref{MDT_GAMMA} and the Mellin inversion theorem, we must have that $w_{1,1} = \mathcal{G}^a$ on $\R_{>0}$. Suppose that $d_2=1$, then 
    $$\MT{w_{1,1}}(s) = \frac{\Gamma(s+a)}{\Gamma(s+b+1)},\quad s\in\Sigma,$$
    Therefore, by Example \ref{MDT_BETA} and the Mellin inversion theorem, we must have that $w_{1,1} = \mathcal{B}^{a,b+1}$ on $(0,1)$. 
\end{proof}

\subsection{Additive setting} \label{EX_A}

In the additive setting, the eigenvalue densities of LUE and GUE matrices are elementary examples of the (multiple) orthogonal polynomial ensembles covered by Theorem \ref{MOPE_ADT}. This follows immediately from Example \ref{ADT_GAMMA} \& \ref{ADT_GAUSSIAN}. The fact that they are of additive derivative type was already observed in \cite[Ex. 2.4]{FKK:PolyaE}.

\begin{cor} \label{OPE_ADT_EX} Suppose that $a>-1$.
\begin{itemize}
    \item[$i)$] If $X\sim\text{LUE}[n;a]$, i.e. $\text{EV}(X)\sim\text{OPE}[n;\mathcal{G}^a]$, then $\text{EV}(X)\sim\text{PE}_{\text{ADT}}[n;\mathcal{G}^{a+n-1}]$.
    \item[$ii)$] If $X\sim\text{GUE}[n]$, i.e. $\text{EV}(X)\sim\text{OPE}[n;\mathcal{G}]$, then $\text{EV}(X)\sim\text{PE}_{\text{ADT}}[n;\mathcal{G}]$.
\end{itemize}
\end{cor}

Since the additive derivative type structure of the eigenvalue densities of independent random matrices is preserved under taking their sums, i.e. if $X_1$ and $X_2$ are independent random matrices with $\text{EV}(X_1)\sim\text{PE}_{\text{ADT}}[n;\om_1]$ and $\text{EV}(X_2)\sim\text{PE}_{\text{ADT}}[n;\om_2]$, then $\text{EV}(X_1+X_2)\sim\text{PE}_{\text{ADT}}[n;\om_1\ast_{\LT{}}\om_2]$, see \cite{FKK:PolyaE}, the above can be used to obtain other examples of polynomial ensembles of derivative type as well. It turns out that adding LUE matrices to an LUE matrix only changes the value of the parameter in the underlying weight and adding GUE matrices to a GUE matrix only dilates the underlying weight. The specific structure of the Laplace transform of the underlying weight therefore remains unchanged. In order to go beyond this class, we will introduce a dilated version of the LUE, denoted by $\text{LUE}[n;a,c]$, with eigenvalue density $\text{OPE}[n;\mathcal{D}_c\mathcal{G}^{a}]$, in terms of the shift operator $(\mathcal{D}_c f) =f(x/c)/c$ and $c>0$. Note that by Example \ref{ADT_GAMMA}, we have
    $$ \LT{[\mathcal{D}_c\mathcal{G}^{a}]}(s)  = \frac{\Gamma(a+1)}{(1+cs)^{a+1}}. $$
The influence of a similarly dilated version of the GUE, denoted by $\text{GUE}[n;c]$, with eigenvalue density $\text{OPE}[n;\mathcal{D}_{\sqrt{2c}}\mathcal{G}]$, is less substantial. Note that by Example \ref{ADT_GAUSSIAN}, we have
    $$ \LT{[\mathcal{D}_{\sqrt{2c}}\mathcal{G}]}(s)  = \sqrt{\pi} \exp\left(cs^2\right),\quad \text{Re}(s)>0. $$
Theorem \ref{MOPE_ADT} shows that the eigenvalue densities of sums of these dilated LUE and GUE matrices are multiple orthogonal polynomial ensembles.

\begin{cor} \label{MOPE_ADT_EX}
    Let $n\in\N$ and take $\vec{n}\in\mathcal{S}^r$ with $\sz{n}=n$. Suppose that all $a_i+n_i-n>-1$ and that all $c_i>0$ are distinct. Let 
    $X_i\sim\text{GUE}[n;c_i]$, for $i\in\{1,\dots,q_0\}$,
    $ X_i\sim\text{LUE}[n;a_i+n_i-n,c_i]$, for $i\in\{q_0+1,\dots,q_1\}$, 
    be independent random matrices. Then $\text{EV}(X_1+\dots+X_{q_1})\sim\text{PE}_{\text{ADT}}[n;\om_{2,n}]$
    with
        $$ \LT{\om_{2,n}}(s) = \frac{\exp\left((\sum_{i=1}^{q_0} c_i)s^2\right)}{\prod_{i=q_0+1}^{q_1} (1+c_i s)^{a_i+n_i}},\quad \text{Re}(s)>0, $$
    and $\text{EV}(X_1+\dots+X_{q_1})\sim\text{MOPE}[n;w_{2,1},\dots,w_{2,q_1-q_0+1}]$ with
        $$ \LT{w_{2,j}}(s) = \frac{\exp\left((\sum_{i=1}^{q_0} c_i)s^2\right)}{\prod_{i=q_0+1}^{q_1} (1+c_i s)^{a_i}} \frac{s^{j-1}}{\prod_{i=q_0+1}^{\max\{q_0+j,q_1\}} (1+c_i s)},\quad \text{Re}(s)>0.$$
\end{cor}

Recall that in the multiplicative setting, most of the ensembles covered by Theorem \ref{MOPE_MDT} arose as products of certain elementary ensembles (related to JUE and LUE matrices). This could be explained by the fact that the underlying functions $\om_{1,n}$ in \eqref{MOPE_MDT_om} could be decomposed into the functions associated to the elementary ensembles using the Mellin convolution. Remarkably, for the functions $\om_{2,n}$ in \eqref{MOPE_ADT_om}, a similar property in terms of the Laplace convolution doesn't hold. Hence, in the additive setting, the ensembles described in Theorem \ref{MOPE_ADT} have a much richer structure.
\medbreak

It turns out that in order to construct most of the functions $\om_{2,n}$ in \eqref{MOPE_ADT_om}, we require the following generalizations of the elementary weights. 

\begin{defi} \label{ADT_EX_BF} \phantom{}
    \begin{itemize}
        \item[$i)$] Let $d\in\Z_{\geq 0}$ and $c>0$. We define
        $$\text{Ai}_{d}^{c}(x) = \int_{1+i\R} \exp\left( (-1)^{\lfloor \frac{d}{2} \rfloor} c s^{d+2} + sx \right) \frac{ds}{2\pi i},\quad x\in\R. $$
        \item[$ii)$] Let $d\in\Z_{\geq 0}$, $a>-1$, $b\in\R_{\neq 0}$ and $c>0$. We define
    $$\text{Be}_{d}^{a,b,c}(x) = c^{-1} e^{-x/c} \sum_{k=0}^\infty \frac{b^k (x/c)^{dk+a}}{k! \Gamma(dk+a+1)},\quad x\in\R_{>0}.$$
    \end{itemize}
\end{defi}

The functions in $i)$ generalize the normal density $\mathcal{G}$ and Airy function $\text{Ai}$ since 
$$\text{Ai}_0^c(x) = \frac{1}{2\sqrt{\pi c}} \mathcal{G}(\frac{x}{\sqrt{2c}}) ,\quad \text{Ai}_1^c(x) = \frac{1}{(3c)^{\frac{1}{3}}} \text{Ai}(-\frac{x}{(3c)^{\frac{1}{3}}}),\quad x\in\R.$$ 
The functions $x\mapsto \text{Ai}_{2n+1}^1(-x)$ already appeared in \cite{HZ:Ai_gen}, where they are called higher-order Airy functions. The functions in $ii)$ generalize the gamma density $\mathcal{G}^a$ and variants of the $I$-Bessel function $I_a$ since
$$\text{Be}_0^{a,b,c}(x) = \frac{e^b}{\Gamma(a+1)}\mathcal{G}^a(x/c),\quad \text{Be}_1^{a,b,c}(x) = \left(\frac{x}{bc}\right)^{\frac{a}{2}} e^{-x/c} I_{a}(2\sqrt{bx/c}) ,\quad x\in\R_{>0}.$$
The functions $x\mapsto e^{x}\text{Be}_d^{a-1,1,1}(x)$ already appeared in \cite[Eq. 10.46.1]{DLMF}, where they are called generalized Bessel functions (or Wright functions).
\medbreak

Using the Laplace inversion theorem and the Fubini-Tonelli theorem, it is straightforward that show that the functions $\text{Ai}_{d}^c$ and $\text{Be}_{d}^{a,b,c}$ are in $L^1_{\LT{},\R_{>0}}(\R)$ and that their Laplace transforms are given by
    $$\LT{\text{Ai}_{d}}(s) = \exp\left( (-1)^{\lfloor \frac{d}{2} \rfloor} c s^{d+2} \right),\quad \LT{\text{Be}_{d}^{a,b,c}}(s) = \dfrac{\exp\left(\frac{b}{(1+cs)^{d}}\right)}{(1+cs)^{a+1}} \quad,\quad \text{Re}(s)>0.$$
\medbreak

The result below shows that at least some of these functions give rise a \textit{positive} polynomial ensemble of additive derivative type. To our knowledge, the associated ensembles have not been studied before.

\begin{prop}
    The function $\text{Be}_{1}^{a,b,c}$ defines an $n$-point polynomial ensemble of additive derivative type for all $n<a+2$.
\end{prop}
\begin{proof}
    It suffices to show this for $c=1$ as $\text{Be}_{1}^{a,b,c} = \mathcal{D}_c\text{Be}_{1}^{a,b,1}$. To do so, we will employ \cite[\S 3.2: Thm. 2.1]{Karlin}, where it was shown that functions of the form
    $$ f_a(x) = \begin{dcases}
        \sum_{k=0}^{\infty} c_k \frac{x^{k+a}}{\Gamma(k+a+1)},\quad x>0, \\
        0,\quad x\leq 0,
    \end{dcases}$$
    are (additive) Pólya frequency functions of any order $n<a+2$, whenever $(c_k)_{k\in\N}$ satisfies $0<\sum_{k=0}^{\infty} c_k <\infty$ and $\det[c_{i-j}]_{i,j=1}^n\geq 0$ ($c_k:=0$ for $k<0$). In a subsequent remark, it was further noted that these conditions hold for $c_k=1/k!$. % and that one can extend to any $a\in\Z_{\geq 0}$ using \cite[\S 3.2: Lem. 2.3]{Karlin}.
    The desired result then follows from \cite[Thm. 2.9 (1)]{FKK:PolyaE}.
\end{proof}

We can now slightly extend Corollary \ref{MOPE_ADT_EX}.

\begin{cor}
    Let $n\in\N$ and take $\vec{n}\in\mathcal{S}^r$ with $\sz{n}=n$. Suppose that all $a_i+n_i-n>-1$, all $b_i\in\R_{\neq 0}$ and that all $c_i>0$ are distinct. Let 
    $X_i\sim\text{GUE}[n;c_i]$, for $i\in\{1,\dots,q_0\}$,
    $ X_i\sim\text{LUE}[n;a_i+n_i-n,c_i]$, for $i\in\{q_0+1,\dots,q_1\}$, and 
    $X_i$ with $\text{EV}(X_j)\sim\text{PE}_{\text{ADT}}(n;\text{Be}_{1}^{a_i+n_i-1,b_i,c_i})$, for $i\in\{q_1+1,\dots,q_2\}$,
    be independent random matrices. Then $\text{EV}(X_1+\dots+X_{q_2})\sim\text{PE}_{\text{ADT}}[n;\om_{2,n}]$
    with
        $$ \LT{\om_{2,n}}(s) = \frac{\exp\left((\sum_{i=1}^{q_0} c_i)s^2 + \sum_{i=q_1+1}^{q_2} \frac{b_i}{1+c_i s}\right)}{\prod_{i=q_0+1}^{q_2} (1+c_i s)^{a_i+n_i}},\quad \text{Re}(s)>0, $$
    and $\text{EV}(X_1+\dots+X_{q_2})\sim\text{MOPE}[n;w_{2,1},\dots,w_{2,q_2-q_0+1}]$ with
        $$ \LT{w_{2,j}}(s) = \frac{\exp\left((\sum_{i=1}^{q_0} c_i)s^2 + \sum_{i=q_1+1}^{q_2} \frac{b_i}{1+c_i s}\right)}{\prod_{i=q_0+1}^{q_2} (1+c_i s)^{a_i}} \frac{s^{j-1}}{\prod_{i=q_0+1}^{\max\{q_0+j,q_2\}} (1+c_is)},\quad \text{Re}(s)>0.$$

\end{cor}

Similarly as in the multiplicative setting, it turns out that many of the functions $\om_{2,n}$ in \eqref{MOPE_ADT_om} are not admissible, i.e. they don't belong to $L^1_{\LT{},\Sigma}(\R)$ or don't lead to a \textit{positive} ensemble. A first step into this direction is the following in which we examine some constraints induced by the underlying space $L_{\LT{},\Sigma}^1(\R)$.

\begin{prop} \label{ADT_COND}
In the setting of Theorem \ref{MOPE_ADT}, we have $\sum_{i=1}^r d_1(i) \text{Im}\,a_i = 0$. If we write
$$ \exp\left(c \int_{s_0}^s \prod_{i=1}^r \frac{(t+a_i)^{d_1(i)}}{(t+b_i)^{d_2(i)}} dt \right) =  \frac{\exp\left(\sum_{i=1}^{d} \alpha_i s^{k+1} + O(s)\right)}{\prod_{i=1}^r (s+b_i)^{d_2(i)\beta_i}},\quad s\to\infty,$$
in terms of $\alpha_i,\beta_i\in\R$, then
\begin{itemize} 
    \item[$i)$] $(-1)^{\lfloor \frac{d+1}{2} \rfloor} \alpha_d<0$ if $d \geq 1$ and $d\equiv_2 1$,
    \item[$ii)$] $ (-1)^{\lfloor \frac{d+1}{2} \rfloor} (s(d+1)\alpha_{d}+\alpha_{d-1})<0$ for all $s\in\Sigma$ if $d \geq 1$ and $d\equiv_2 0$,
    \item[$iii)$] $\sum_{i=1}^r \beta_i > - 1$ if $d=0$.
\end{itemize}
\end{prop}
\begin{proof}
    Since $\LT{v_1}$ is real-valued, the $a_i\in\C\backslash\R$ must appear in conjugate pairs (consider, e.g., $(\LT{v_1})'(s_0)$), hence after a partial fraction decomposition all $\alpha_k\in\R$. We will use Proposition \ref{LT_RL}, i.e.
    $$ \lim_{|t|\to\infty} \LT{v_1}(s+it)=0, $$
    for every $s\in\Sigma$. Suppose that $d\geq 1$, then the asymptotics of $|\LT{v_1}(s+it)|$ are determined by the exponential part in the above representation. The stated conditions then follow from the fact that
        $$ \text{Re}\left(\sum_{k=1}^{d} \alpha_{k} (s+it)^{k+1} \right) = \begin{dcases}
            (s(d+1)\alpha_{d}+\alpha_{d-1}) t^{d} (1+o(1)) ,\quad d \equiv_4 0,\\
            - \alpha_{d} t^{d+1} (1+o(1)) ,\quad d \equiv_4 1, \\
            -(s(d+1)\alpha_{d}+\alpha_{d-1}) t^{d} (1+o(1)) ,\quad d \equiv_4 2,\\
            \alpha_{d} t^{d+1} (1+o(1)) ,\quad d \equiv_4 3.\\
        \end{dcases} $$
    Suppose that $d=0$, then the asymptotics of $|\LT{v_1}(s+it)|$ are determined by the rational part in the above representation. Since $d=0$, we must have all $d_2(i)=1$, and hence $\sum_{i=1}^{r} \beta_i + 1  > 0$.
\end{proof}

\textit{Positivity} of the ensemble also seems to induce certain conditions on the parameters of the functions $\om_{2,n}$ in \eqref{MOPE_ADT_om}. For example, several functions of the form $\text{Ai}_{d}^{c}$ are known to oscillate on part of the real line so that $\om_{2,n}$ can not be of this form. In fact, it seems that whenever $\LT{\om_{2,n}}$ contains of factor of the form $\LT{\text{Ai}_{d}^{c}}$ with $d\geq 1$, the function $\om_{2,n}$ will have some oscillations on the real line. Evidence of this may be provided by applying a steepest descent analysis on its inverse Laplace transform. We will not explore this here as it is rather technical to describe the conditions on the parameters under which such an asymptotic analysis can be carried out. As mentioned in the introduction, it would be interesting to have an optimal set of conditions for which $\om_{2,n}$ in \eqref{MOPE_ADT_om} gives rise to a \textit{positive} ensemble of additive derivative type, because of the connection to (additive) Pólya frequency functions of finite order.
\medbreak

By considering the simplest setting covered by Theorem \ref{MOPE_ADT}, i.e. $r=1$, we can show that the eigenvalue densities of LUE and GUE matrices are essentially the only orthogonal polynomial ensembles of additive derivative type. This is the content of Proposition \ref{ADT_OPE}, which we will prove now.

\begin{proof}[Proof of Proposition \ref{ADT_OPE}]
    The fact that the eigenvalue densities of LUE and GUE matrix are orthogonal polynomial ensembles of additive derivative type was already observed in Corollary \ref{OPE_ADT_EX}. By Theorem \ref{MOPE_ADT}, the latter are characterized by a weight $w_{2,1}\in L^{1}_{\LT{},\Sigma}(\R)$ with a Laplace transform of the form
        $$ \LT{w_{2,1}}(s) = \exp\left(c \int_{s_0}^s \frac{(t+a)^{d_1}}{(t+b)^{d_2}} dt \right) \frac{1}{(s+b)^{d_2}},\quad s\in\Sigma, $$
    where $a\in\C$, $b,c,s_0\in\R$ and $d_1,d_2\in\{0,1\}$ with $\max\{d_1,d_2\}=1$. Suppose that $d_2=0$, then
    $$ \LT{w_{2,1}}(s) = \exp\left(\alpha_0+\alpha_1 s + \alpha_2 s^2 \right),\quad s\in\Sigma. $$
    with $\alpha_0,\alpha_1\in\R$ and $\alpha_2>0$ according to Proposition \ref{ADT_COND}. We can then assume that $\alpha_0=\alpha_1=0$ and $\alpha_2=1/4$, otherwise we can apply an appropriate scaling and affine transformation to $w_{2,1}$. In that case, $\LT{w_{2,1}}(s) = \exp\left(s^2/4\right)$ for $s\in\Sigma$ and thus $w_{2,1}=\mathcal{G}/\sqrt{\pi}$ on $\R$ by Example \ref{ADT_GAUSSIAN} and the Laplace inversion theorem. Suppose that $d_2=1$, then
    $$ \LT{w_{2,1}}(s) = \frac{\exp\left(\alpha_0+\alpha_1 s\right)}{(s+b)^{\beta+1}},\quad s\in\Sigma, $$
    with $\alpha_0,\alpha_1\in\R$, $-b\not\in\Sigma$ and $\beta>-1$ according to Proposition \ref{ADT_COND}. We can then assume that $\alpha_0=\alpha_1=0$ and $b=1$, otherwise we can apply an appropriate scaling and affine transformation to $w_{2,1}$. In that case, $\LT{w_{2,1}}(s) = 1/(s+1)^{\beta+1}$ for $s\in\Sigma$ (as $s>0$) and thus $w_{2,1}=\mathcal{G}^{\beta}/\Gamma(\beta+1)$ on $\R_{>0}$ by Example \ref{ADT_GAMMA} and the Laplace inversion theorem.
\end{proof}

\section{Finite free probability} \label{FFP}

The goal of this section is two-fold. First, we will provide a partial answer to an open problem posed by Martínez-Finkelshtein at the Hypergeometric and Orthogonal Polynomials Event in Nijmegen (May, 2024), see \cite{HOPE}, related to orthogonality of the finite free multiplicative and additive convolution of polynomials from free probability. We will do this by interpreting the problem in terms of random matrices and by making use of the notion of functions (or polynomial ensembles) of multiplicative and additive derivative type. Second, we will explain how in a more restricted setting the associated multiple orthogonal polynomials come into play. We will show that such polynomials (de)compose naturally under the finite free multiplicative and additive convolution.

\subsection{Multiplicative setting}

We first recall the notion of finite free multiplicative convolution, see \cite[Def. 1.4]{MSS:fin_free_conv}.

\begin{defi}
    The finite free multiplicative convolution of two polynomials $p_1(x) = \sum_{k=0}^n p_1[k] x^k $ and $p_2(x) = \sum_{k=0}^n p_2[k] x^k $ of degree at most $n$ is defined as
        $$ (p_1 \boxtimes_n p_2)(x) = \sum_{k=0}^n \frac{p_1[k]p_2[k]}{(-1)^{n-k} \binom{n}{k}} x^k . $$
\end{defi}

The first part of the open problem involves orthogonality of the finite free multiplicative convolution of two polynomials that satisfy given orthogonality relations. More precisely, one considers polynomials $p_{1,n}$ and $p_{2,n}$ of degree $n$ that satisfy
    $$\int_0^\infty p_{j,n}(x) q_{j,k}(x) dx = 0,\quad k\in\{0,\dots,n-1\},\quad j\in\{1,2\},$$
for some functions $(q_{1,k})_{k=0}^{n-1}$ and $(q_{2,k})_{k=0}^{n-1}$ and then asks whether their finite free multiplicative convolution also satisfies
    $$\int_0^\infty (p_{1,n}\boxtimes_n p_{2,n})(x) q_k(x) dx = 0,\quad k\in\{0,\dots,n-1\},$$
for some explicit functions $(q_k)_{k=0}^{n-1}$.
\medbreak

We will interpret this problem in terms of products of random matrices using the result below, which is essentially a reformulation of \cite[Thm. 1.5]{MSS:fin_free_conv} in terms of random matrices instead of deterministic ones. Note that it is essentially the finite version of the free convolution law for products of random matrices, i.e. the fact that, if $\mu_j$ is the limiting eigenvalue distribution of $X_j$, the eigenvalue distribution of $X_1X_2$ tends to the free multiplicative convolution $\mu_1\boxtimes\mu_2$ as $n\to\infty$, whenever one of the underlying matrix ensembles is invariant under unitary conjugation.

\begin{prop} \label{MC_AP}
    Let $X_1$ and $X_2$ be independent $n\times n$ normal random matrices and assume that one of the matrix ensembles is invariant under unitary conjugation. Then, 
    $$\mathbb{E}[\det(xI_n-X_1X_2)] = \mathbb{E}[\det(xI_n-X_1)] \boxtimes_n \mathbb{E}[\det(xI_n-X_2)].$$
    Consequently, also 
    $$\mathbb{E}[\det(xI_n-(X_1X_2)(X_1X_2)^\ast)] = \mathbb{E}[\det(xI_n-X_1X_1^\ast)] \boxtimes_n \mathbb{E}[\det(xI_n-X_2X_2^\ast)].$$
\end{prop}
\begin{proof}
    Since $X_1$ and $X_2$ are independent random matrices, the expectation on the left hand side is given by
    $$\mathbb{E}[\det(xI_n-X_1X_2)] = \iint \det(xI_n-X_1X_2) d\nu_1(X_1)d\nu_2(X_2). $$
    It was shown in \cite[Thm. 1.5]{MSS:fin_free_conv} that for normal matrices $X_1$ and $X_2$, one has
    $$\int_{U(n)} \det(xI_n-X_1QX_2Q^\ast) d\nu_{\text{Haar}}(Q) = \det(xI_n-X_1) \boxtimes_n \det(xI_n-X_2).$$
    Since at least one of the underlying matrix ensembles is invariant under unitary conjugation, we thus obtain
    $$\mathbb{E}[\det(xI_n-X_1X_2)] = \iint\det(xI_n-X_1) \boxtimes_n \det(xI_n-X_2) d\nu_1(X_1)d\nu_2(X_2). $$
    Now we can use bilinearity of the finite free multiplicative convolution to obtain the desired result.
    For the second claim, we can assume without loss of generality that the ensemble associated to $X_1$ is invariant under unitary conjugation, otherwise we can swap the order of $X_1$ and $X_2$ via the identity
        $$\mathbb{E}[\det(xI_n-(X_1X_2)(X_1X_2)^\ast)] = \mathbb{E}[\det(xI_n-(X_1X_2)^\ast(X_1X_2))].$$
    In that case, we can use the first result to write the expectation on the left as $$\mathbb{E}[\det(xI_n-X_1)] \boxtimes_n (\mathbb{E}[\det(xI_n-X_2X_2^\ast)] \boxtimes_n \mathbb{E}[\det(xI_n-X_1^\ast)]). $$
    After using the fact that the finite free multiplicative convolution is associative and commutative, this can be written as
    $$(\mathbb{E}[\det(xI_n-X_1)] \boxtimes_n \mathbb{E}[\det(xI_n-X_1^\ast)]) \boxtimes_n \mathbb{E}[\det(xI_n-X_2X_2^\ast)], $$
    We can then apply the first result again.
\end{proof}

It turns out that for certain random matrices, similar decompositions also occur on a deeper level. This is straightforward to see from the results in \cite{KK:SV_EV,KK:mult_conv} after making the proper identifications. Given an $n\times n$ random matrix $X$, we will use the notation $\text{SSV}(X)\sim\text{PE}[n]$, resp. $\text{SSV}(X)\sim\text{PE}_{\text{MDT}}[n]$, to denote that the squared singular values of $X$ follow a polynomial ensemble, resp. polynomial ensemble of multiplicative derivative type.

\begin{prop} \label{MC_BS}
Let $X_1$ and $X_2$ be independent random matrices with $\text{SSV}(X_1)\sim\text{PE}[n]$ and $\text{SSV}(X_2)\sim\text{PE}_{\text{MDT}}[n;\om]$. Let $(p_{k,j},q_{k,j})_{j=0}^{n-1}$ be the biorthogonal system for the correlation kernel of $\text{SSV}(X_k)$. Then the biorthogonal system $(p_k,q_k)_{j=0}^{n-1}$ for the correlation kernel of $\text{SSV}(X_1X_2)$ is given by
    $$ p_j = p_{1,j} \boxtimes_{j} p_{2,j},\quad q_j = q_{1,j} \ast_{\MT{}} \om,\quad j\in\{0,\dots,n-1\}. $$
\end{prop}
\begin{proof}
Denote $p_{1,j}(x) = \sum_{k=0}^{j} p_{1,j}[k] x^k$. It was shown in \cite[Cor. 3.7]{KK:mult_conv} that 
$$ p_j(x) = \sum_{k=0}^j p_{1,j}[k] \frac{(\mathcal{M}\om)(j+1)}{(\mathcal{M}\om)(k+1)} x^k ,\quad q_{j} = q_{1,j} \ast_{\MT{}} \om_n. $$
On the other hand, according to \cite[Lem. 4.2]{KK:SV_EV}, one has
$$ p_{2,j}(x) = \sum_{k=0}^j (-1)^{j-k} \binom{j}{k} \frac{(\mathcal{M}\om)(j+1)}{(\mathcal{M}\om)(k+1)} x^k. $$
It then remains to use the definition of $\boxtimes_{j}$.
\end{proof}

We will now provide a partial answer to the previously mentioned open problem. Let $X_1$ and $X_2$ be independent $n\times n$ normal random matrices for which 
    $$ p_{j,n}(x) = \mathbb{E}[\det(xI_n-X_jX_j^\ast)],\quad j\in\{1,2\}. $$
In order to encode the orthogonality conditions through the random matrices, it is natural to assume that $\text{SSV}(X_1)\sim \text{PE}[n;q_{1,0},\dots,q_{1,n-1}]$ and $\text{SSV}(X_2)\sim \text{PE}[n;q_{2,0},\dots,q_{2,n-1}]$. Indeed, in that case, we automatically have
    $$\int_0^\infty p_{j,n}(x) q_{j,k}(x) dx = 0,\quad k\in\{0,\dots,n-1\},\quad j\in\{1,2\}.$$
With this set-up, Proposition \ref{MC_AP} allows us to describe the finite free multiplicative convolution of $p_{1,n}$ and $p_{2,n}$ explicitly:
    $$ (p_{1,n}\boxtimes_n p_{2,n})(x) = \mathbb{E}[\det(xI_n-(X_1X_2)(X_1X_2)^\ast)].  $$
Again, if we would like $p_{1,n}\boxtimes_n p_{2,n}$ to satisfy 
    $$\int_0^\infty (p_{1,n}\boxtimes_n p_{2,n})(x) q_k(x) dx = 0,\quad k\in\{0,\dots,n-1\},$$
it is natural to demand that $\text{SSV}(X_1X_2)\sim \text{PE}[n;q_0,\dots,q_{n-1}]$. The problem is then to determine when exactly this occurs. This is exactly the problem that drove the research in random matrix theory surrounding products of random matrices that ultimately led to the notion of multiplicative derivative type. We now know that whenever one of the initial polynomial ensembles is of multiplicative derivative, the squared singular values of the product also follow a polynomial ensemble. The associated functions are known explicitly, see \cite[Cor. 3.7]{KK:mult_conv} (or Proposition \ref{MC_BS} above): if $q_{2,0},\dots,q_{2,n-1}$ are of multiplicative derivative type w.r.t. $\omega$, then $q_k = q_{1,k} \ast_{\MT{}} \om$ for $k\in\{0,\dots,n-1\}$. 
\medbreak

In the more restrictive setting of the problem in which the polynomials $p_{1,n}$ and $p_{2,n}$ are assumed to be multiple orthogonal polynomials, one of the initial polynomial ensembles ends up being a multiple orthogonal polynomial ensemble of multiplicative derivative type. These are exactly the objects that we characterized in Theorem \ref{MOPE_MDT}. Following the proof of \cite[Lem. 4.2]{KK:SV_EV}, we can determine the associated multiple orthogonal polynomials explicitly.

\begin{prop} \label{MDT_MOP}  
     Let $w_1,\dots,w_r\in L^1_{\MT{},\Sigma}(\R_{>0})$ with $\Z_{\geq 1}\subset \Sigma$. If, for all $m\in\{1,\dots,n\}$, the functions in $\mathcal{W}(m;w_1,\dots,w_r)$ are of multiplicative derivative type, then the type II polynomial w.r.t. $(w_1,\dots,w_r)$ for the multi-index $\vec{n}\in\mathcal{S}^r$ with $\sz{n}=n$ is given by
    \begin{equation} \label{MDT_MOP_FORM}
        P_{\vec{n}}(x) = \sum_{k=0}^{\sz{n}} (-1)^{\sz{n}-k} \binom{\sz{n}}{k} \frac{(\MT{\om_{1,n}})(\sz{n}+1)}{(\MT{\om_{1,n}})(k+1)} x^k,
    \end{equation}
    where $\MT{\om_{1,n}}$ is as in \eqref{MOPE_MDT_om}.
\end{prop}
\begin{proof}
    Note that $P_{\vec{n}}$ is a monic polynomial of degree $\sz{n}$. We have to show that it satisfies the orthogonality conditions
        $$ \int_0^\infty P_{\vec{n}}(x) x^k w_j(x) dx = 0,\quad k\in\{0,\dots,n_j-1\},\quad j\in\{1,\dots,r\}. $$
    By Theorem \ref{MOPE_MDT_CORE_MOM}, this is equivalent to showing that
        $$ \int_0^\infty P_{\vec{n}}(x) (D^{k}\om_{1,n})(x) dx = 0,\quad k\in\{0,\dots,n-1\}, $$
    where $(Df)(x)=-xf'(x)$ and $n=\sz{n}$. To do so, we will write
        $$ P_{\sz{n}}(x) = (-1)^n n! \int_{\mathcal{C}} \frac{\MT{\om_{1,n}}(n+1)}{\MT{\om_{1,n}}(t+1)}  \frac{x^t}{(-t)_{n+1}} \frac{dt}{2\pi i},  $$
    in terms of a counterclockwise contour $\mathcal{C}$ in $\Sigma+i\R$ enclosing $[0,n]$ once. We can do this because $\MT{\om_{1,n}}(t)\neq 0$ for $\text{Re}(t)\in\Sigma$ and $\Z_{\geq 1}\subset\Sigma$. After interchanging the integrals, we obtain
    $$\int_0^\infty P_n(x) (D^{k}\om_{1,n})(x) dx = (-1)^n n! \int_{\mathcal{C}} \frac{\MT{\om_{1,n}}(n+1)}{\MT{\om_{1,n}}(t+1)}  \frac{(\MT{D^k\om_{1,n}})(t+1)}{(-t)_{n+1}} \frac{dt}{2\pi i}. $$
    By Proposition \ref{MT_DIFF}, we have $(\MT{D^k\om_{1,n}})(t)=t^k \MT{\om_{1,n}}(t)$ for $t\in\Sigma+i\R$ and thus
    $$\int_0^\infty P_{\vec{n}}(x) (D^{k}\om_{1,n})(x) dx = (-1)^n n! (\mathcal{M}\om_{1,n})(n+1) \int_{\mathcal{C}} \frac{(t+1)^k}{(-t)_{n+1}} \frac{dt}{2\pi i}.$$
    Since $k\in\{0,\dots,n-1\}$, the integrand is $O(|t|^{-2})$ as $|t|\,\to\infty$, and thus by blowing up the contour, we can show that the integral vanishes.   
\end{proof}

The polynomials in \eqref{MDT_MOP_FORM} have some nice (de)composition properties in terms of the finite free multiplicative convolution: if we denote them by $P_{\vec{n}}(x;\om_{1,n})$, then we have $P_{\vec{n}}(x;\om_{1,n}) = P_{\vec{n}}(x;\om_{1,n}^1) \boxtimes_{\sz{n}} P_{\vec{n}}(x;\om_{1,n}^2)$ whenever $\om_{1,n}=\om_{1,n}^1\ast_{\MT{}}\om_{1,n}^2$. The systems of weights to which Proposition \ref{MDT_MOP} applies are described in \eqref{MOPE_MDT_w}.

\subsection{Additive setting}

We start by recalling the notion of finite free additive convolution, see \cite[Def.~1.1]{MSS:fin_free_conv}.

\begin{defi}
    The finite free additive convolution of two polynomials $p_1$ and $p_2$ of degree at most $n$ is defined as
        $$ (p_1 \boxplus_n p_2)(x) = \frac{1}{n!} \sum_{k=0}^n p_1^{(k)}(x) p_2^{(n-k)}(0) . $$
\end{defi}

The second part of the open problem involves orthogonality of the finite free additive convolution of two polynomials that satisfy given orthogonality relations. More precisely, one consider polynomials $p_{1,n}$ and $p_{2,n}$ of degree $n$ that satisfy
    $$\int_{-\infty}^\infty p_{j,n}(x) q_{j,k}(x) dx = 0,\quad k=0,\dots,n-1,$$
for some functions $(q_{1,k})_{k=0}^{n-1}$ and $(q_{2,k})_{k=0}^{n-1}$ and then asks whether their finite free additive convolution also satisfies
    $$\int_{-\infty}^\infty (p_{1,n}\boxplus_n p_{2,n})(x) q_k(x) dx = 0,\quad k\in\{0,\dots,n-1\},$$
for some explicit functions $(q_k)_{k=0}^{n-1}$.
\medbreak

In this setting, we will interpret the problem in terms of sum of random matrices instead of products. To this end, we will make use of the result below, which is essentially a reformulation of \cite[Thm. 1.2]{MSS:fin_free_conv} in terms of random matrices instead of deterministic ones. Observe that it is essentially the finite version of the free convolution law for sums of random matrices, i.e. the fact that, if $\mu_j$ is the limiting eigenvalue distribution of $X_j$, the eigenvalue distribution of $X_1+X_2$ tends to the free additive convolution $\mu_1\boxplus\mu_2$ as $n\to\infty$, whenever one of the underlying matrix ensembles is invariant under unitary conjugation.

\begin{prop} \label{AC_AP}
    Let $X_1$ and $X_2$ be independent $n\times n$ normal random matrices and assume that one of the matrix ensembles is invariant under unitary conjugation. Then, 
    $$\mathbb{E}[\det(xI_n-(X_1+X_2))] = \mathbb{E}[\det(xI_n-X_1)] \boxplus_n \mathbb{E}[\det(xI_n-X_2)].$$
\end{prop}
\begin{proof}
    The idea of the proof is the same as in the multiplicative setting (Proposition \ref{MC_AP}). The only difference is that we have to use the fact that
    $$\int_{U(n)} \det(xI_n-(X_1+QX_2Q^\ast)) d\nu_{\text{Haar}}(Q) = \det(xI_n-X_1) \boxplus_n \det(xI_n-X_2),$$
    for normal $n\times n$ matrices $X_1$ and $X_2$, see \cite[Thm. 1.2]{MSS:fin_free_conv}.
\end{proof}

Again, for certain random matrices, similar decompositions occur on a deeper level as well. This essentially follows from the results in \cite{Kieburg:add_conv} after making the proper identifications (and using the Laplace transform instead of the Fourier transform).

\begin{prop} \label{AC_BS}
Let $X_1$ and $X_2$ be independent random matrices with $\text{EV}(X_1)\sim\text{PE}[n]$ and $\text{EV}(X_2)\sim\text{PE}_{\text{ADT}}[n;\om]$. Let $(p_{k,j},q_{k,j})_{j=0}^{n-1}$ be the biorthogonal system for the correlation kernel of $\text{EV}(X_k)$. Then the biorthogonal system $(p_{j},q_{j})_{j=0}^{n-1}$ for the correlation kernel of $\text{EV}(X_1+X_2)$ is given by
    $$ p_j = p_{1,j} \boxplus_{j} p_{2,j},\quad q_j = q_{1,j} \ast_{\LT{}} \om,\qquad j=0,\dots,n-1. $$
\end{prop}
\begin{proof}
According to \cite[Thm. III.8]{Kieburg:add_conv}, one has
    $$ p_j(x) = \frac{1}{(n-1)!} \int_0^\infty e^{-r} \left(r-\frac{d}{dt}\frac{d}{dx}\right)^{n-1}\left[\frac{p_{1,j}(x)}{\LT{\om}(t)}\right]_{t=0} dr,\quad  q_j =  q_{1,j} \ast_{\LT{}} \omega. $$
On the other hand, it was shown in \cite[Thm. III.1]{Kieburg:add_conv} that
    $$ p_{2,j}(x) = \frac{1}{j!}\left(x-\frac{d}{dt}\right)^j\left[\frac{1}{\LT{\om}(t)}\right]_{t=0}. $$
After expanding the $(n-1)$-th power in the integral representation for $p_j(x)$, we obtain
$$\begin{aligned}
    p_j(x) = \sum_{k=0}^{n-1} \frac{1}{k!(n-1-k)!} \left(-\frac{d}{dt}\frac{d}{dx}\right)^{k} \left[\frac{p_{1,j}(x)}{\LT{\om}(t)}\right]_{t=0} \int_0^\infty e^{-r} r^{n-1-k} dr,
\end{aligned}$$
and thus 
    $$ p_j(x) = \sum_{k=0}^{j} \frac{(-1)^k}{k!} \left(\frac{1}{\LT{\om}}\right)^{(k)}(0) (p_{1,j})^{(k)}(x).$$
Now use the fact that     
$$p_{2,j}(x) = \sum_{k=0}^j \frac{(-1)^k}{k!(j-k)!} \left(\frac{1}{\LT{\om}}\right)^{(k)}(0) x^{j-k}, $$
and the definition of $\boxplus_j$.
\end{proof}

Following a similar strategy as in the multiplicative setting, we will now provide a partial answer to the previously mentioned open problem. Let $X_1$ and $X_2$ be independent $n\times n$ normal random matrices for which 
    $$ p_{j,n}(x) = \mathbb{E}[\det(xI_n-X_j)],\quad j\in\{1,2\}. $$
In order to encode the orthogonality conditions through the underlying random matrices, it is again natural to assume that $\text{EV}(X_1)\sim \text{PE}[n;q_{1,0},\dots,q_{1,n-1}]$ and $\text{EV}(X_2)\sim \text{PE}[n;q_{2,0},\dots,q_{2,n-1}]$, because then we automatically have
    $$\int_{-\infty}^\infty p_{j,n}(x) q_{j,k}(x) dx = 0,\quad k\in\{0,\dots,n-1\},\quad j\in\{1,2\}.$$
By Proposition \ref{AC_AP}, the finite free additive convolution of $p_{1,n}$ and $p_{2,n}$ is given by
    $$ (p_{1,n}\boxtimes_n p_{2,n})(x) = \mathbb{E}[\det(xI_n-(X_1+X_2)].  $$
Again, if we would like $p_{1,n}\boxtimes_n p_{2,n}$ to satisfy 
    $$\int_{-\infty}^\infty (p_{1,n}\boxtimes_n p_{2,n})(x) q_k(x) dx = 0,\quad k\in\{0,\dots,n-1\},$$
it is natural to demand that $\text{EV}(X_1+X_2)\sim \text{PE}[n;q_0,\dots,q_{n-1}]$. The problem is then to determine when exactly this occurs, which is exactly the problem that drove the research in random matrix theory surrounding sums of random matrices that ultimately led to the notion of additive derivative type. We now know that whenever one of the initial polynomial ensembles is of additive derivative, the eigenvalues of the sum also follow a polynomial ensemble. The associated functions are known explicitly, see \cite[Thm. III.8]{Kieburg:add_conv} (or Proposition \ref{AC_BS} above): if the functions $q_{2,0},\dots,q_{2,n-1}$ are of additive derivative type w.r.t. $\omega$, then $q_k = q_{1,k} \ast_{\LT{}} \om$ for $k\in\{0,\dots,n-1\}$. 
\medbreak

Whenever the polynomials $p_{1,n}$ and $p_{2,n}$ in the open problem are assumed to be multiple orthogonal polynomials, one of the initial polynomial ensembles ends up being a multiple orthogonal polynomial ensemble of additive derivative type. These are exactly the objects that we characterized in Proposition \ref{MOPE_ADT}. In the following result, we will use some ideas in the proof of \cite[Thm. III.1]{Kieburg:add_conv} to determine the associated multiple orthogonal polynomials explicitly. Note that the assumption that all $\LT{w_j}$ are analytic at the origin is the analogue of the condition $\Z_{\geq 1}\subset\Sigma$ in the multiplicative setting (see Proposition \ref{MDT_MOP}) as it ensures that all the moments of $w_j$ exist.

\begin{prop} \label{ADT_MOP}  
     Let $w_1,\dots,w_r\in L^1_{\LT{},\Sigma}(\R)$ and suppose that the $\LT{w_j}$ are analytic at the origin. If, for all $m\in\{1,\dots,n\}$, the functions in $\mathcal{W}(n;w_1,\dots,w_r)$ are of additive derivative type, then the type II polynomial w.r.t. $(w_1,\dots,w_r)$ for the multi-index $\vec{n}\in\mathcal{S}^r$ with $\sz{n}=n$ is given by
    \begin{equation} \label{ADT_MOP_FORM}
        P_{\vec{n}}(x) = \LT{\om_{2,n}}(0) \sum_{k=0}^{\sz{n}} (-1)^{\sz{n}-k} \binom{\sz{n}}{k} \left(\frac{1}{\LT{\om_{2,n}}}\right)^{(\sz{n}-k)}(0)\, x^k,
    \end{equation}
    where $\LT{\om_{2,n}}$ is as in \eqref{MOPE_ADT_om}.
\end{prop}
\begin{proof}
    Note that $P_{\vec{n}}$ is a monic polynomial of degree $\sz{n}$. We have to show that it satisfies the orthogonality conditions
        $$ \int_{-\infty}^\infty P_{\vec{n}}(x) x^k w_j(x) dx = 0,\quad k\in\{0,\dots,n_j-1\},\quad j\in\{1,\dots,r\}, $$
    By Theorem \ref{MOPE_ADT_CORE}, this is equivalent to showing that
        $$ \int_0^\infty P_{\vec{n}}(x) \om_{2,n}^{(k)}(x) dx = 0,\quad k\in\{0,\dots,n-1\}. $$
    To do so, we will write
        $$ P_{\vec{n}}(x) = (-1)^n n! \int_{\mathcal{C}} \frac{\LT{\om_{2,n}}(0)}{\LT{\om_{2,n}}(t)}  \frac{e^{-tx}}{t^{n+1}} \frac{dt}{2\pi i},  $$
    in terms of a counterclockwise contour $\mathcal{C}$ enclosing $0$ once. We can do this because $\LT{\om_{2,n}}$ is analytic at the origin as 
    $$\LT{\om_{2,n}}\in\spn{(\LT{w_j)^{(k-1)}} \mid k=1,\dots,n_j,j=1,\dots,r}.$$ 
    After interchanging the integrals, we obtain
    $$\int_{-\infty}^\infty P_n(x) \om_{2,n}^{(k)}(x) dx = (-1)^n n! \int_{\mathcal{C}} \frac{\LT{\om_{2,n}}(0)}{\LT{\om_{2,n}}(t)}  \frac{(\LT{\om_{2,n}^{(k)})(t)}}{t^{n+1}} \frac{dt}{2\pi i}. $$
    By Proposition \ref{LT_DIFF}, we have $(\LT{\om_{2,n}^{(k)}})(t)=t^k \LT{\om_{2,n}}(t)$ for $t\in\Sigma+i\R$ and thus
    $$\int_{-\infty}^\infty P_n(x) \om_{2,n}^{(k)}(x) dx = (-1)^n n! \LT{\om_{2,n}}(0) \int_{\mathcal{C}} \frac{1}{t^{n-k+1}} \frac{dt}{2\pi i}.$$
    The integral therefore vanishes for all $k\in\{0,\dots,n-1\}$.   
\end{proof}

The polynomials in \eqref{ADT_MOP_FORM} have some nice (de)composition properties in terms of the finite free additive convolution: if we denote them by $P_{\vec{n}}(x;\om_{2,n})$, then we have $P_{\vec{n}}(x;\om_{2,n}) = P_{n}(x;\om_{2,n}^1) \boxplus_{\sz{n}} P_{\vec{n}}(x;\om_{2,n}^2)$ whenever $\om_{2,n}=\om_{2,n}^1\ast_{\LT{}}\om_{2,n}^2$. The systems of weights to which Proposition \ref{ADT_MOP} applies are described in \eqref{MOPE_ADT_w}.

%%%---------------------------------
\section{Proof of Theorem \ref{MOPE_MDT} \& \ref{MOPE_ADT}} 

In this section, we will prove Theorem \ref{MOPE_MDT} \& \ref{MOPE_ADT}. We will do so in several steps. In order to prove Theorem \ref{MOPE_MDT}, structurally, we have to show the following. Recall that
    $$ \mathcal{W}(n;w_1,\dots,w_r) = \{ x\mapsto x^{k-1} w_j(x) \mid k=1,\dots,n_j,j=1,\dots,r \}.$$

\begin{comment}
    \begin{thm} \label{MDT_MR}
    Let $\vec{w}$ be a system of $r$ real-valued and a.e. positive weights in $ L^1_{\MT{},\Sigma}(\R_{>0})$ with $\Z_{\geq 1}\subset \Sigma$. Then the following are equivalent:
    \begin{itemize}
        \item[$i)$] for all $\vec{n}\in\mathcal{S}^r$, the set $\mathcal{W}_{\vec{n}}(\vec{w})$ is of multiplicative derivative type w.r.t. some function $\om_{\vec{n}}$,
        \item[$ii)$] there exists an invertible upper-triangular matrix $U$ such that the weights in $\vec{v} = \vec{w} U$ have a Mellin transform of the form
    \begin{equation} \label{MDT_MR_FORM}
        \MT{v}_j(s) = c^s \, \prod_{i=1}^r \frac{\Gamma(s+a_i)^{d_1(i)}}{\Gamma(s+b_i)^{d_2(i)}} \frac{s^{j-1}}{\prod_{i=1}^j (s+b_i)^{d_2(i)}}, \quad s\in\Sigma,
    \end{equation}
    for some $a_i\in\C$, $b_i\in\R$, $c\in\R$ and $d_1,d_2:\{1,\dots,r\}\to\{0,1\}$ with $d_1=1$ or $d_2=1$ on $\{1,\dots,r\}$.
    \end{itemize}
    In that case, for all $\vec{n}\in\mathcal{S}^r$,
    $$ \MT{\om}_{\vec{n}}(s) = \prod_{i=1}^r \frac{\Gamma(s+a_i)^{d_1(i)}}{\Gamma(s+b_i+n_i)^{d_2(i)}},\quad s\in\Sigma.$$
\end{thm}
\end{comment}

\begin{thm} \label{MOPE_MDT_CORE}
    Let $v_1,\dots,v_n\in L^1_{\MT{},\Sigma}(\R_{>0})$. Then the following are equivalent:
    \begin{itemize}
        \item[$i)$] $\Z_{\geq 1}\subset\Sigma$ and there exists positive $w_1,\dots,w_r\in L^1_{\MT{},\Sigma}(\R_{>0})$ such that 
        $$\spn{v_j}_{j=1}^n=\text{span}\,\mathcal{W}(n;w_1,\dots,w_r),$$ 
        and, for all $m\in\{1,\dots,r+1\}$, the functions in $\mathcal{W}(m;w_1,\dots,w_r)$ are of multiplicative derivative type,
        \item[$ii)$] $\spn{v_j}_{j=1}^n=\text{span}\,\mathcal{W}(n;w_{1,1},\dots,w_{1,r})$ with $w_{1,1},\dots,w_{1,r}$ as in \eqref{MOPE_MDT_w},
        \item[$iii)$] $v_1,\dots,v_n$ are of multiplicative derivative type w.r.t. $\om_{1,n}$ as in \eqref{MOPE_MDT_om}.
    \end{itemize}
\end{thm}

The fact that we have to assume positivity of the weights is an analytic obstruction caused by the use of the Mellin transform. By dropping this condition, we can essentially only recover the moments of the weights as shown in the result below.

\begin{thm} \label{MOPE_MDT_CORE_MOM}
    Suppose that $w_1,\dots,w_r\in L^1_{\MT{},\Sigma}(\R_{>0})$ are such that the functions in $\mathcal{W}(m;w_1,\dots,w_r)$ are of multiplicative derivative type for all $m\in\{1,\dots,n\}$. Then:
    \begin{itemize}
        \item[$i)$] $\text{span}\,\mathcal{W}(m;w_1,\dots,w_r)=\text{span}\,\mathcal{W}(n;w_{1,1},\dots,w_{1,r})$ where $w_{1,1},\dots,w_{1,r}$ are such that \eqref{MOPE_MDT_w} holds for all $s\in\Sigma\cap\Z_{\geq 1}$,
        \item[$ii)$] the functions in $\mathcal{W}(m;w_1,\dots,w_r)$ are of multiplicative derivative type w.r.t. $\om_{1,n}$ where $\om_{1,n}$ is such that \eqref{MOPE_MDT_om} holds for all $s\in\Sigma\cap\Z_{\geq 1}$.
    \end{itemize}
\end{thm}

In order to prove Theorem \ref{MOPE_ADT}, structurally, we have to show the following. 

\begin{thm} \label{MOPE_ADT_CORE}
    Let $v_1,\dots,v_n\in L^1_{\LT{},\Sigma}(\R)$. Then the following are equivalent:
    \begin{itemize}
        \item[$i)$] there exists $w_1,\dots,w_r\in L^1_{\LT{},\Sigma}(\R)$ such that 
        $$\spn{v_j}_{j=1}^n=\text{span}\,\mathcal{W}(n;w_1,\dots,w_r),$$
        and, for all $m\in\{1,\dots,r+1\}$, the functions in $\mathcal{W}(m;w_1,\dots,w_r)$ are of additive derivative type,
        \item[$ii)$] $\spn{v_j}_{j=1}^n=\text{span}\,\mathcal{W}(n;w_{2,1},\dots,w_{2,r})$ with $w_{2,1},\dots,w_{2,r}$ as in \eqref{MOPE_ADT_w},
        \item[$iii)$] $v_1,\dots,v_n$ are of additive derivative type w.r.t. $\om_{2,n}$ as in \eqref{MOPE_ADT_om}.
    \end{itemize}
\end{thm}

We will prove Theorem \ref{MOPE_MDT_CORE_MOM}, \ref{MOPE_MDT_CORE} \& \ref{MOPE_ADT_CORE} in several steps. First, in Section \ref{DT},  we will give an algebraic description of functions of multiplicative/additive derivative type. Afterwards, in Section \ref{CLASS}, we will restrict to the functions that appear in multiple orthogonal polynomial ensembles, i.e. the functions in 
$\mathcal{W}(n;w_1,\dots,w_r)$. The additional structure will allow us to identify the Mellin/Laplace transform of $w_1,\dots,w_r$. This will show $i)$ in Theorem \ref{MOPE_MDT_CORE_MOM} and the implication from $i)$ to $ii)$ in Theorem \ref{MOPE_MDT_CORE} \& \ref{MOPE_ADT_CORE}. It then remains to prove the equivalence between $ii)$ and $iii)$ for all $n\in\N$ (structurally, this is what we need to show for Proposition \ref{MOPE_MDT_w_om}). This is because $ii)$ and $iii)$ together imply $i)$ as the first part of $i)$ follows immediately from $ii)$ and the second part is a consequence of the previous result. This then also shows $ii)$ in Theorem \ref{MOPE_MDT_CORE_MOM}.

\subsection{Algebraic description of derivative type} \label{DT}

Crucial to our approach are the two theorems below, which can be seen as an alternative definition of collections of functions of multiplicative/additive derivative type. They essentially allow us to convert the analytic conditions in Definition~\ref{MDT_def}/\ref{ADT_def} into algebraic ones by moving to the Mellin/Laplace space. These results also motivate the precise analytic conditions that appear in these definitions.

\begin{thm} \label{MDT_ALG}
    Suppose that $v_1,\dots,v_n\in L_{\MT{},\Sigma}^1(\R_{>0})$. Then the following are equivalent:
    \begin{itemize}
        \item[$i)$] $v_1,\dots,v_n$ are of multiplicative derivative type w.r.t. some function $\om:\R_{>0}\to\C$,
        \item[$ii)$] there exists a function $\hat{\om}:\Sigma\to\C$ such that
        $$\spn{\MT{v}_j(s)}_{j=1}^n = \spn{ s^{j-1} \hat{\om}(s)}_{j=1}^n,\quad s\in \Sigma.$$
    \end{itemize}
    In that case, $\MT{\om}=\hat{\om}$ on $\Sigma$.
\end{thm}
\begin{proof}
    Denote $(Df)(x)=-xf'(x)$ for functions $f:\R_{>0}\to\C$. For the implication from $i)$ to $ii)$, we know that $\om\in C^{n-2}(\R_{>0})$, with $\om^{(n-2)} \in AC_{\text{loc}}(\R_{>0})$, is such that a.e.
        $$\spn{v_j}_{j=1}^n = \spn{D^{j-1}\om}_{j=1}^n.$$
    If we apply Proposition \ref{MT_DIFF} to $\nu_j = D^{j-1}\om\in\spn{v_j}_{j=1}^n\subset L_{\MT{},\Sigma}^1(\R_{>0})$, we obtain
        $$ \MT{\nu}_j(s) = s^{j-1} \MT{\om}(s) ,\quad s\in\Sigma, $$        
    and thus
        $$\spn{\MT{\nu}_j(s)}_{j=1}^n = \spn{ s^{j-1} \hat{\om}(s)}_{j=1}^n,\quad s\in \Sigma.$$
    In particular, $\dim\spn{\MT{\nu}_j(s)}_{j=1}^n = n$. Together with the fact that
        $$\spn{\MT{\nu}_j(s)}_{j=1}^n \subset \spn{\MT{v}_j(s)}_{j=1}^n,\quad s\in \Sigma,$$
    this leads to desired result. \\
    For the implication from $ii)$ to $i)$, we take $\nu_j\in\spn{v_j}_{j=1}^n\subset L_{\MT{},\Sigma}^1(\R_{>0})$ such that $\MT{\nu}_j(s) = s^{j-1} \hat{\om}(s)$ for all $s\in\Sigma$. Since $\MT{\nu}_n(s) = s^{n-1} \MT{\nu}_1(s)$ for all $s\in\Sigma$, it follows from Proposition \ref{MT_DIFF} that there exists $\om\in C^{n-2}(\R_{>0})$, with $\om^{(n-2)} \in AC_{\text{loc}}(\R_{>0})$, such that a.e. $\nu_1=\om$ and $\nu_n=D^{n-1}\om$. We will show by induction on $j\in\{1,\dots,n-1\}$ that also a.e. $\nu_j = D^{j-1}\om$. In that case, the functions $v_1,\dots,v_n$ will be of multiplicative derivative type because
    $$ \spn{\nu_j}_{j=1}^n \subset \spn{v_j}_{j=1}^n,$$
    and $\dim \spn{\nu_j}_{j=1}^n = n$ (if it would be less, we can take the Mellin transform and get a contradiction). For $j=1$, there is nothing to prove. Suppose that $\nu_j = D^{j-1}\om$ for some $j\in\{1,\dots,n-1\}$, then $ \MT{\nu_{j+1}}(s) = s \MT{[D^{j-1}\om]}(s)$ for all $s\in\Sigma$. Proposition \ref{MT_DIFF} then implies that there exists $F\in AC_{\text{loc}}(\R_{>0})$ such that a.e. $F=D^{j-1}\om$ and $\nu_{j+1}=DF$. Since both $D^{j-1}\om,F\in C^{0}(\R_{>0})$, we must have that $F=D^{j-1}\om$ on $\R_{>0}$ and thus a.e. $\nu_{j+1}=D^{j}\om$. 
\end{proof}

\begin{thm} \label{ADT_ALG}
    Suppose that $v_1,\dots,v_n\in L_{\LT{},\Sigma}^1(\R)$. Then the following are equivalent:
    \begin{itemize}
        \item[$i)$] $v_1,\dots,v_n$ is of additive derivative type w.r.t. some function $\om:\R\to\C$,
        \item[$ii)$] there exists a function $\hat{\om}:\Sigma\to\C$ such that
        $$\spn{\LT{v}_j(s)}_{j=1}^n = \spn{ s^{j-1} \hat{\om}(s)}_{j=1}^n,\quad s\in \Sigma.$$
    \end{itemize}
        In that case, $\LT{\om}=\hat{\om}$ on $\Sigma$.
\end{thm}
\begin{proof}
   We can prove this using the same ideas as in the multiplicative setting, but using the Laplace transform instead of the Mellin transform and Proposition \ref{LT_DIFF} instead of Proposition \ref{MT_DIFF}.
\end{proof}

A side product of the previous results is the following connection between functions of additive and multiplicative derivative type.

\begin{prop}
    Consider the operator $\tilde{f}: \R_{>0}\to\C:x\mapsto f(-\ln x)$ on $f:\R\to\C$. The functions $v_1,\dots,v_n\subset L_{\LT{},\Sigma}^1(\R)$ are of additive derivative type if and only if the functions $\tilde{v}_1,\dots,\tilde{v}_n\subset L_{\MT{},\Sigma}^1(\R_{>0})$ are of multiplicative derivative type.      
\end{prop}
\begin{proof}
    This follows immediately from the algebraic conditions in Theorem~\ref{MDT_ALG}~\&~\ref{ADT_ALG} and the fact that $ \LT{f} = \MT{\tilde{f}}$ on $\Sigma$.
\end{proof}

In what follows, we will describe some common properties of functions of multiplicative and additive derivative type, based on the algebraic description in Theorem~\ref{MDT_ALG} \& \ref{ADT_ALG}. As such, we will state all the results in terms of a linear operator $\T$ acting on functions on some set $X\subset\R$. In the next section, we will restrict again to $\T\in\{\MT{},\LT{}\}$.
\medbreak

We will first describe a necessary condition for consecutive collections of functions to be of multiplicative/additive derivative type. The previous algebraic description reveals the following structure on the transform of the underlying derivative type functions.

\begin{prop} \label{DT_ALG_FUN}
    Let $v_1,\dots,v_n:X\to\C$ and consider a linear operator $\T$ that acts on $v\in \spn{v_j}_{j=1}^n$ and $\T{v}:\Sigma\to\C$ is defined on some infinite set $\Sigma\subset\C$ and is non-vanishing. If, for all $m\in\{1,\dots,n\}$, there exists a function $\hat{\om}_m:\Sigma\to\R$ such that
        $$\spn{\T{v_j}(s)}_{j=1}^{m} = \spn{ s^{j-1} \hat{\om}_m(s)}_{j=1}^{m},\quad s\in\Sigma,$$
    then there exists $b_0,\dots,b_{n-1}\in\R$ and a function $d:\{1,\dots,n-1\}\to\{0,1\}$ such that, for all $m\in\{1,\dots,n\}$,
    $$ \hat{\om}_m(s) = \frac{b_0 \T{v_1}(s)}{\prod_{i=1}^{m-1} (s+b_i)^{d(i)}},\quad s\in\Sigma.$$
\end{prop}
\begin{proof}
    We will show this by induction on $m\in\{1,\dots,n\}$. For $m=1$, this is immediate. Let $m\in\{1,\dots,n-1\}$ and suppose that the formula holds for all integers $<m$. By the assumption, we know that for every $\nu\in \spn{v_j}_{j=1}^{m}$, there exists a polynomial $p_{m-1}(s;\nu)$ of degree at most $m-1$ such that
        $$ \T{\nu}(s) = p_{m-1}(s;\nu) \hat{\om}_{m}(s),\quad s\in\Sigma. $$
    Note that both $\T{\nu},\hat{\om}_{m}\in\spn{\T{v_j}}_{j=1}^{m}$ so that they are non-vanishing on $\Sigma$. In that case, the same holds for $\Sigma\to\C:s\mapsto p_{m-1}(s;\nu)$. Hence, for every $\nu\in \spn{v_j}_{j=1}^{m}$, we must have that
        $$ \hat{\om}_{m}(s) = \frac{\T{\nu}(s)}{p_{m-1}(s;\nu)},\quad s\in\Sigma. $$
    By considering appropriate functions $\nu$, we can establish the desired formula. Consider first $\nu_1\in \spn{v_j}_{j=1}^{m-1}$ with $ \T{\nu_1}(s) = s^{d_0} \prod_{i=1}^{m-2} (s+b_i)^{d(i)} \hat{\om}_{m-1}(s) $ where $d_0 = \#\{1\leq i\leq n-2 : d(i) = 0 \}$. Since the degree of $s^{d_0} \prod_{i=1}^{m-2} (s+b_i)^{d(i)}$ is at most $m-2$, such $\nu_1$ exists by the assumption. By the induction hypothesis, we then have
        $$ \T{\nu_1}(s) = b_0 s^{d_0} \T{v_1}(s),\quad s\in\Sigma, $$
    and thus
        $$ \hat{\om}_{m}(s) = \frac{b_0 s^{d_0} \T{v_1}(s)}{p_{m-1}(s;\nu_1)},\quad s\in\Sigma, $$
    for infinitely many $s\in\Sigma$. On the other hand, by the assumption, we can also consider $\nu_2\in \spn{v_j}_{j=1}^{m-1}$ with $ \T{\nu_2}(s) = \hat{\om}_{m-1}(s) $. The induction hypothesis then implies that    
        $$ \T{\nu_2}(s) = \frac{b_0 \T{w_1}(s)}{\prod_{i=1}^{m-2}(s+b_i)^{d(i)}},\quad s\in\Sigma, $$
    and thus
        $$ \hat{\om}_{m}(s) = \frac{b_0 \T{v_1}(s)}{p_{m-1}(s;\nu_2) \prod_{i=1}^{m-2}(s+b_i)^{d(i)}},\quad s\in\Sigma. $$
    By combining the two identities, we obtain
        $$ p_{m-1}(s;\nu_1) = p_{m-1}(s;\nu_2) s^{d_0} \prod_{i=1}^{m-2}(s+b_i)^{d(i)},\quad s\in\Sigma.$$ 
    In fact, since $\Sigma$ contains infinitely many points, we can view this as a formal identity between the polynomials. As the degree of $s^{d_0} \prod_{i=1}^{m-2}(s+b_i)^{d(i)}$ is exactly $m-2$, the degree of $p_{m-1}(s;\nu_2)$ is at most $1$. We can therefore write $p_{m-1}(s;\nu_2)=(s+b_{n-1})^{d(m-1)}$ so that
        $$ \hat{\om}_{m}(s) = \frac{\T{v_1}(s)}{\prod_{i=1}^{m-1}(s+b_i)^{d(i)}}, \quad s\in\Sigma, $$
    as desired.
\end{proof}

The previous result also allows us to uncover the structure of the transforms of the initial functions $v_1,\dots,v_n$.

\begin{prop} \label{DT_ALG}
    Let $v_1,\dots,v_n:X\to\C$ and consider a linear operator $\T$ that acts on $v\in \spn{v_j}_{j=1}^n$ and $\T{v}:\Sigma\to\C$ is defined on some infinite set $\Sigma\subset\C$ and is non-vanishing. Then the following are equivalent: 
    \begin{itemize}
        \item[$i)$] for all $m\in\{1,\dots,n\}$, there exists a function $\hat{\om}_m:\Sigma\to\C$ such that
        $$\spn{\T{v_j}(s)}_{j=1}^{m} = \spn{ s^{k-1} \hat{\om}_m(s)}_{k=1}^{m},\quad s\in\Sigma,$$
        \item[$ii)$] there exists $p_0,\dots,p_{N-1}\in\R[s]$, $b_1,\dots,b_{N-1}\in\R$ and $d:\{1,\dots,N-1\}\to\{0,1\}$ such that, for all $n\in\{1,\dots,N\}$, such that, for all $m\in\{1,\dots,n\}$,
        $$ \frac{\T{v_m}(s)}{\T{v_1}(s)} = \frac{p_{m-1}(s)}{\prod_{i=1}^{m-1}(s+b_i)^{d(i)}},\quad s\in\Sigma, $$
    and $\max\{\deg p_{m-1}(s),\deg \prod_{i=1}^{m-1}(s+b_i)^{d(i)}\}=m-1$.
    \end{itemize}
    In that case, for all $m\in\{1,\dots,n\}$,
    $$ \T{\om_m}(s) = \frac{\T{v_1}(s)}{\prod_{i=1}^{m-1} (s+b_i)^{d(i)}},\quad s\in\Sigma.$$
\end{prop}
\begin{proof}
    We will first prove the implication from $i)$ to $ii)$. It follows from Proposition \ref{DT_ALG_FUN} that there exists $b_0,\dots,b_{n-1}\in\R$ and a function $d:\{1,\dots,n-1\}\to\{0,1\}$ such that, for all $m\in\{1,\dots,n\}$,
    $$ \hat{\om}_m(s) = \frac{b_0 \T{v_1}(s)}{\prod_{i=1}^{m-1} (s+b_i)^{d(i)}},\quad s\in\Sigma.$$
    In that case, for all $m\in\{1,\dots,n\}$, there exists a polynomial $p_{m-1}\in\R[s]$ of degree at most $m-1$ such that
        $$ \frac{\T{v_m}(s)}{\T{v_1}(s)} = \frac{p_{m-1}(s)}{\prod_{i=1}^{m-1}(s+b_i)^{d(i)}},\quad s\in\Sigma.$$
    We have to show that the degree conditions are satisfied. Consider the smallest value $m\in\{1,\dots,n\}$ for which 
    $$\max\{\deg p_{m-1}(s),\deg \prod_{i=1}^{m-1}(s+b_i)^{d(i)}\}<m-1.$$ 
    In that case, $\deg p_{m-1} \leq m-2$ and $\prod_{i=1}^{m-1}(s+b_i)^{d(i)} = \prod_{i=1}^{m-2}(s+b_i)$. A partial fraction decomposition then gives
    $$ \frac{\T{v_m}(s)}{\T{v_1}(s)} = \sum_{j=1}^{m-1} \frac{c_j}{\prod_{i=1}^{j-1}(s+b_i)^{d(i)}},\quad s\in\Sigma. $$
    This implies that $\T{v_m} \in \spn{\T{v_j}}_{j=1}^{m-1}$, which contradicts the fact that 
    $$\dim\spn{\T{v_j}(s)}_{j=1}^{m}=\dim\spn{s^{k-1}\hat{\om}_m(s)}_{j=1}^{m}=m.$$
    
    We will now show that the implication from $ii)$ to $i)$ holds. We will show $i)$ by induction on $m\in\{1,\dots,n\}$ using the stated expression for $\om_m$. For $m=1$, the desired result is immediate. Let $m\in\{1,\dots,n\}$ and suppose that $i)$ holds for all integers $<m$. We have to show that
        $$\spn{\T{v_j}(s)}_{j=1}^{m} = \spn{ s^{k-1} \hat{\om}_m(s)}_{k=1}^{m},\quad s\in\Sigma.$$
    We will first prove the inclusion from left to right. Let $v\in\spn{v_j}_{j=1}^m$ so that $ \T{v} = \sum_{j=1}^m c_j \T{v}_j $. Then by the induction hypothesis and $ii)$, we can write
        $$ \T{v}(s) = \T{v_1}(s) \left( \frac{q_{m-2}(s)}{\prod_{i=1}^{m-2} (s+b_i)^{d(i)}} + c_m \frac{p_{m-1}(s)}{\prod_{i=1}^{m-1} (s+b_i)^{d(i)}}\right),\quad s\in\Sigma, $$
    in terms of some polynomial $q_{m-2}(s)$ of degree at most $m-2$. After writing both terms on the common denominator $\prod_{i=1}^{m-1} (s+b_i)^{d(i)}$, we can conclude that $\T{v}(s)\in\spn{s^{k-1}\hat{\om}_m(s)}_{k=1}^m$. \\    
    For the implication from right to left, let $q_{m-1}(s)$ be a polynomial of degree at most $m-1$ and consider $\T{v}(s) = q_{m-1}(s) \hat{\om}_m(s)$. We will show that there exists $c\in\R$ such that 
    $$(q_{m-1}(s)-cp_{m-1}(s))\hat{\om}_m(s) \in\spn{s^{k-1}\hat{\om}_{m-1}(s)}_{k=1}^{m-1},\quad s\in\Sigma.$$ 
    In that case, $\T{v} \in\spn{\T{v_j}}_{j=1}^{m}$, because  
    $$(q_{m-1}(s)-cp_{m-1}(s))\hat{\om}_m(s)\in \spn{\T{v_j}(s)}_{j=1}^{m-1},\quad s\in\Sigma,$$
    by the induction hypothesis, and $ p_{m-1}(s)\hat{\om}_m(s) = \T{v_m}(s) $ by $ii)$. Suppose first that $d(m-1)=1$ and thus 
    $$\hat{\om}_m(s)=\frac{\hat{\om}_{m-1}(s)}{s+b_{m-1}},\quad s\in\Sigma.$$
    Since $p_{m-1}(-b_{m-1}) \neq 0$ (otherwise the degree condition would be violated), we can take $c\in\R$ such that $q_{m-1}(s)-cp_{m-1}(s)$ vanishes at $s=-b_{m-1}$. In that case, we have
    $$(q_{m-1}(s)-cp_{m-1}(s))\hat{\om}_m(s) \in\spn{s^{k-1}\hat{\om}_{m-1}(s)}_{k=1}^{m-1}.$$
    Suppose now that $d(m-1)=0$ and thus 
    $$\hat{\om}_{m}=\hat{\om}_{m-1},\quad s\in\Sigma.$$ 
    Under this condition we have $\deg p_{m-1}=m-1$ (otherwise the degree condition would be violated) and we can take $c\in\R$ such that the degree of $q_{m-1}(s)-cp_{m-1}(s)$ is at most $m-2$. In that case, we have 
    $$(q_{m-1}(s)-cp_{m-1}(s))\hat{\om}_m(s) \in\spn{s^{k-1}\hat{\om}_{m-1}(s)}_{k=1}^{m-1},$$
    as desired.
\end{proof}

By taking appropriate linear combinations of $v_1,\dots,v_n$, we can slightly simplify the structure of the corresponding ratios $\T{v_n}/\T{v_1}$.

\begin{prop} \label{DT_ALG_COMP}  
    In the setting of Proposition \ref{DT_ALG}, there exists an invertible upper-triangular matrix $U$ such that the weights defined by $(v_{0,j})_{j=1}^{n} = (v_j)_{j=1}^{n} U$ satisfy, for all $m\in\{1,\dots,n\}$,
        $$ \frac{\T{v_{0,m}}(s)}{\T{v_{0,1}}(s)} = \frac{s^{n-1}}{\prod_{i=1}^{m-1}(s+b_i)^{d(i)}},\quad s\in\Sigma.$$
\end{prop}
\begin{proof}
    It follows immediately from Proposition \ref{DT_ALG} that there exists $v_{0,m}\in\spn{v_j}_{j=1}^m$ such that $\T{v_{0,m}}(s) = s^{m-1} \hat{\om}_m(s)$ for all $s\in\Sigma$. Then we can use the stated expression for $\hat{\om}_m(s)$.
\end{proof}

With the result below, we aim to provide some more intuition on Proposition~\ref{DT_ALG}. It essentially states that the $m$-dependence in $\hat{\om}_m(s)$ can be removed if besides multiplying $\T{v_1}$ by $s$, we are also allowed to multiply with some $1/(s+b_j)$.

\begin{prop} \label{DT_ALG_CONSEC}
    In the setting of Proposition \ref{DT_ALG}, there exists an invertible upper-triangular matrix $U$ such that the weights defined by $(v_{0,j})_{j=1}^{n} = (v_j)_{j=1}^{n} U$ satisfy, for all $m\in\{1,\dots,n\}$,
        $$ \frac{\T{v}_{m+1}(s)}{\T{v}_m(s)} = \begin{dcases}
            s,\quad d(m)=0,\\
            \frac{1}{s+b_m},\quad d(m) = 1,
        \end{dcases}\qquad s\in\Sigma. $$
\end{prop}
\begin{proof}
    By Proposition \ref{DT_ALG_COMP}, we can assume that 
        $$ \frac{\T{v}_{0,m+1}(s)}{\T{v}_{0,m}(s)} = \frac{s}{(s+b_n)^{d(n)}}.$$
    If $d(m)=0$, we obtain the first case. If $d(m)=1$, we can do a partial fraction decomposition and subtract $v_{0,m}$ from $v_{0,m+1}$ to obtain the other case.    
\end{proof}

\subsection{Characterization} \label{CLASS}

\subsubsection{Multiplicative setting} \label{CLASS_M}

Using the results from the previous section, we can obtain the following characterization of functions of multiplicative derivative type.

\begin{prop} \label{MDT_SR}
    Let $v_1,\dots,v_n\in L^1_{\MT{},\Sigma}(\R_{>0})$.
    Then the following are equivalent:
    \begin{itemize}
        \item[$i)$] for all $m\in\{1,\dots,n\}$, the functions $v_1,\dots,v_m$ are of multiplicative derivative type w.r.t. some function $\om_m:\R_{>0}\to\C$,
        \item[$ii)$] there exists $p_0,\dots,p_{n-1}\in\R[s]$, $b_1,\dots,b_{n-1}\in\R$ and $d:\{1,\dots,n-1\}\to\{0,1\}$ such that, for all $m\in\{1,\dots,n\}$,
        $$ \frac{\MT{v_m}(s)}{\MT{v_1}(s)} = \frac{p_{m-1}(s)}{\prod_{i=1}^{m-1}(s+b_i)^{d(i)}},\quad s\in\Sigma, $$
    and $\max\{\deg p_{m-1}(s),\deg \prod_{i=1}^{m-1}(s+b_i)^{d(i)}\}=m-1$.
    \end{itemize}
    In that case, for all $m\in\{1,\dots,n\}$,
    $$ \MT{\om_m}(s) = \frac{\MT{v_1}(s)}{\prod_{i=1}^{m-1} (s+b_i)^{d(i)}},\quad s\in\Sigma.$$
\end{prop}
\begin{proof}
    We have to combine Theorem \ref{MDT_ALG} and Proposition \ref{DT_ALG}.
\end{proof}    

By considering the collections of functions $\mathcal{W}(n;w_1,\dots,w_r)$ relevant for multiple orthogonal polynomial ensembles, the above allows us to identify the moments of the first function $w_1$. Using the compatibility relations in Proposition \ref{DT_ALG_COMP}, we can then describe the moments of all the functions $w_1,\dots,w_r$.

\begin{proof}[Proof of Theorem \ref{MOPE_MDT_CORE_MOM}: $i)$]
    Since $r<n$, according to Proposition \ref{MDT_SR}, we can write
        $$ \frac{\MT{w}_1(s+1)}{\MT{w}_1(s)} = c \, \frac{\prod_{i=1}^{r} (s+a_i)^{d_1(i)}}{\prod_{i=1}^{r} (s+b_i)^{d_2(i)}},\quad s\in\Sigma,$$
    for some $a_i\in\C$, $b_i,c\in\R$ and $d_1,d_2:\{1,\dots,r\}\to \{0,1\}$ with all $d_1(i)=1$ or all $d_2(i)=1$. Consecutive use of this identity for $s\in\Z_{\geq 1}$, then leads to
        $$ \MT{w}_1(s) = \MT{w}_1(1) c^{s-1} \frac{\prod_{i=1}^{r} (a_i+1)_{s-1}^{d_1(i)}}{\prod_{i=1}^{r} (b_i+1)_{s-1}^{d_2(i)}},\quad s\in\Sigma\cap\Z_{\geq 1}.$$
    Since $\MT{w}_1(s)$ is defined for all $s\in\Z_{\geq 1}$ and doesn't vanish, we must have $a_i,b_i\not\in\Z_{\leq -1}$ and $c\neq 0$. Hence, there exists $c_0\in\C$, namely
        $$ c_0 = \frac{1}{c} \, \frac{\prod_{j=1}^{r} \Gamma(b_j+1)^{d_2(i)}}{\prod_{j=1}^{r} \Gamma(a_j+1)^{d_1(i)}}, $$
    such that 
        $$ \MT{w}_1(s) = c_0 c^s \frac{\prod_{j=1}^{r} \Gamma(s+a_j)^{d_1(i)}}{\prod_{j=1}^{r} \Gamma(s+b_j)^{d_2(i)}},\quad s\in\Sigma\cap\Z_{\geq 1}.$$
    Proposition \ref{DT_ALG_COMP} then implies that there exists an invertible upper-triangular matrix $U$ such that the functions $(w_{1,j})_{j=1}^r = (w_j)_{j=1}^r U$ satisfy
    $$ \MT{w_{1,j}}(s) = c^s \frac{\prod_{j=1}^{r} \Gamma(s+a_j)^{d_1(i)}}{\prod_{j=1}^{r} \Gamma(s+b_j)^{d_2(i)}} \frac{s^{j-1}}{\prod_{i=1}^{j-1}(s+b_i)^{d_2(i)}},\quad s\in\Sigma\cap\Z_{\geq 1}.$$
    After the reparametrization $\vec{b}\mapsto (b_2,\dots,b_r,b_1+1)$ and modifying $d_2$ accordingly, we obtain the desired result.
\end{proof}

In order to describe the Mellin transform on an interval $\Sigma$ and not only at integer values, we have to impose an additional assumption on $w_1,\dots,w_r$. This is an analytic obstruction that also appears in connection to the Bohr-Mollerup Theorem, see, e.g., \cite{MZ:B-M}. The idea is that the moments can't detect functions of the form $e^{\sin(\pi s)}$ as part of the Mellin transform. In line with that theorem, we decide to demand positivity of on $w_1,\dots,w_r$. This is a natural assumption, because later we want to use these functions to create an ensemble.
\medbreak

We are now ready to prove the implication from $i)$ to $ii)$ in Theorem \ref{MOPE_MDT_CORE}.

\begin{proof}[Proof of Theorem \ref{MOPE_MDT_CORE}: $i)\Rightarrow ii)$]
    Since $r<n$, Proposition \ref{MDT_SR} gives rise to a first order difference equation for $\MT{w_1}$ of the form
        $$ \MT{w}_1(s+1) = c \, \frac{\prod_{i=1}^{r} (s+a_i)^{d_1(i)}}{\prod_{i=1}^{r} (s+b_i)^{d_2(i)}} \MT{w}_1(s),\quad s\in\Sigma,$$
     for some $a_i\in\C$, $b_i,c\in\R$ and $d_1,d_2:\{1,\dots,r\}\to \{0,1\}$ with all $d_1(i)=1$ or all $d_2(i)=1$. We will now verify that Proposition \ref{MT_DE_uni} is applicable so that, up to a scalar multiplication, the above difference equation has a unique solution of the form $\MT{w}_1$ with a.e. $w_1>0$. Note that $\Sigma=(s_0,\infty)$ for some $s_0\in[0,1)$ because of the assumption that $\Z_{\geq 1}\subset \Sigma$. Denote 
        $$ g(s) =  c \, \frac{\prod_{i=1}^{r} (s+a_i)^{d_1(i)}}{\prod_{i=1}^{r} (s+b_i)^{d_2(i)}}. $$
     Since $w_1$ is a.e. positive, $\MT{w_1}>0$ on $\Sigma$ and thus also $g>0$ on $\Sigma$. Moreover, $\lim_{n\to\infty} g(n+1)/g(n) = 1$. The conditions of Proposition \ref{MT_DE_uni} are therefore satisfied. With  Theorem~\ref{MOPE_MDT_CORE_MOM}~$i)$ in mind, it is then straightforward to show that the unique solution is given by
    $$ \MT{w}_1(s) = c_0 c^s \frac{\prod_{j=1}^{r} \Gamma(s+a_j)^{d_1(i)}}{\prod_{j=1}^{r} \Gamma(s+b_j)^{d_2(i)}},\quad s\in\Sigma.$$
    We can then proceed similarly as in the proof of Theorem~\ref{MOPE_MDT_CORE_MOM}~$i)$ to describe the Mellin transform of the other weights.
\end{proof}

The last step in the proof of Theorem \ref{MOPE_MDT_CORE} is to prove the equivalence between $ii)$ and $iii)$ for all $n\in\N$ (structurally, this is what we need to prove for Proposition \ref{MOPE_MDT_w_om}). This also shows $ii)$ of Theorem \ref{MOPE_MDT_CORE_MOM}.

\begin{prop}
Let $n\in\N$. Suppose that the Mellin transforms of $w_{1,1},\dots,w_{1,r}\in L^{1}_{\MT{},\Sigma}(\R_{>0})$ are given by \eqref{MOPE_MDT_w} and the Mellin transform of $\om_{1,n}\in L^{1}_{\MT{},\Sigma}(\R_{>0})$ is given by \eqref{MOPE_MDT_om}. Then the functions in $\mathcal{W}(n;w_{1,1},\dots,w_{1,r})$ are of multiplicative derivative type w.r.t. $\om_{1,n}$.
\end{prop}
\begin{proof}
    We will prove this by induction on $n\in\N$ via Proposition~\ref{MDT_SR}. For $n=1$, we have $\MT{w_{1,1}}(s)/\MT{w_{1,1}}(s)=1$ and $ii)$ is satisfied. Hence, $\mathcal{W}(1;w_{1,1},\dots,w_{1,r})$ is of multiplicative derivative type w.r.t. some $\om_1:\R_{>0}\to\C$ with
        $$ \MT{\om_1}(s)=\MT{w_{1,1}}(s),\quad s\in\Sigma, $$
    and thus $\MT{\om_1}=\MT{\om_{1,1}}$ on $\Sigma$ as desired. Suppose that $\mathcal{W}(m;w_{1,1},\dots,w_{1,r})$ is of multiplicative derivative type w.r.t. $\om_{1,m}$ for all $m<n$ and $n=r\eta+j\geq 2$ with $\eta\in\Z_{\geq 0}$ and $j\in\{0,\dots,r-1\}$. Since $i)$ is satisfied for $m\in\{1,\dots,n-1\}$, we obtain $ii)$ for $m\in\{1,\dots,n-1\}$. We need to show that $ii)$ also holds for $m=n$ so that, according to $i)$, $\mathcal{W}(n;w_{1,1},\dots,w_{1,r})$ is of multiplicative derivative type w.r.t. some $\om_n$. If $\MT{w_{1,j}}(s+\eta)/\MT{w_{1,1}}(s)$ is of the appropriate form, we should then recover $\MT{\om_n}=\MT{\om_{1,n}}$ on $\Sigma$. By definition, we have
        $$ \frac{\MT{w_{1,j}}(s+\eta)}{\MT{w_{1,1}}(s)} = c^\eta \, \frac{\prod_{i=1}^{r} (s+a_i)_\eta^{d_1(i)}}{\prod_{i=1}^{r} (s+b_i)_\eta^{d_2(i)}} \frac{s^{j-1}(s+b_1)^{d_2(1)}}{\prod_{i=1}^j (s+b_i+\eta)^{d_2(i)}},\quad s\in\Sigma. $$
    Note that the factor $(s+b_1)^{d_2(1)}$ appears in both the numerator and denominator. The degree of the numerator and denominator is therefore at most $r\eta+j-1=n-1$. Since all $d_1(i)=1$ or all $d_2(i)=1$, at least one of them is equal to $n-1$. Take $\vec{n}\in\mathcal{S}^r$ with $\sz{n}=n$, then $n_i = \eta+1$ for $i\in\{1,\dots,j\}$ and $n_i = \eta$ for $i\in\{j+1,\dots,r\}$. Hence we can write the denominator as $\prod_{i=1}^r (s+b_i)_{n_i}^{d_2(i)}/(s+b_1)^{d_2(1)}$. We can therefore conclude that $\mathcal{W}(n;w_{1,1},\dots,w_{1,r})$ is of multiplicative derivative type w.r.t. a function $\om_n:\R_{>0}\to\C$ with    
        $$ \MT{\om_n}(s) = \frac{(s+b_1)^{d_2(1)}\MT{w_{1,1}}(s)}{\prod_{i=1}^r (s+b_i)_{n_i}^{d_2(i)}},\quad s\in\Sigma. $$
    Together with the identity
    $$\MT{w_{1,1}}(s) = c^s \frac{\prod_{i=1}^{r} \Gamma(s+a_i)^{d_1(i)}}{\prod_{i=1}^{r} \Gamma(s+b_i)^{d_2(i)}} \frac{1}{(s+b_1)^{d_1(i)}},\quad s\in\Sigma,$$
    this yields $\MT{\om_n}=\MT{\om_{1,n}}$ on $\Sigma$ as desired.
\end{proof}

\subsubsection{Additive setting} \label{CLASS_A}

There is a similar characterization of functions of additive derivative type.

\begin{prop} \label{ADT_SR}
    Let $v_1,\dots,v_n\in L^1_{\LT{},\Sigma}(\R)$. Then the following are equivalent:
    \begin{itemize}
        \item[$i)$] for all $m\in\{1,\dots,n\}$, the functions $v_1,\dots,v_m$ are of additive derivative type w.r.t. some function $\om_m:\R\to\C$,
        \item[$ii)$] there exists $p_0,\dots,p_{n-1}\in\R[s]$, $b_1,\dots,b_{n-1}\in\R$ and $d:\{1,\dots,n-1\}\to\{0,1\}$ such that, for all $m\in\{1,\dots,n\}$,
        $$ \frac{\LT{v_m}(s)}{\LT{v_1}(s)} = \frac{p_{m-1}(s)}{\prod_{i=1}^{m-1}(s+b_i)^{d(i)}},\quad s\in\Sigma, $$
    and $\max\{\deg p_{m-1}(s),\deg \prod_{i=1}^{m-1}(s+b_i)^{d(i)}\}=m-1$.
    \end{itemize}
    In that case, for all $m\in\{1,\dots,n\}$,
    $$ \LT{\om_m}(s) = \frac{\LT{v_1}(s)}{\prod_{i=1}^{m-1} (s+b_i)^{d(i)}},\quad s\in\Sigma.$$
\end{prop}
\begin{proof}
    We have to combine Theorem \ref{ADT_ALG} and Proposition \ref{DT_ALG}.
\end{proof}

By applying the previous result to collections of functions $\mathcal{W}(n;w_1,\dots,w_r)$ relevant for multiple orthogonal polynomial ensembles, we will be able to identify the Laplace transform of the first function $w_1$. Compared to the multiplicative setting, to do so, we won't have to impose any additional assumptions on $w_1,\dots,w_r$. The Laplace transforms of the remaining functions $w_2,\dots,w_r$ will follow from the compatibility conditions in Proposition~\ref{DT_ALG_COMP}. This is exactly the content of the implication from $i)$ to $ii)$ in Theorem \ref{MOPE_ADT_CORE}.

\begin{proof}[Proof of Theorem \ref{MOPE_ADT_CORE}: $i)\Rightarrow ii)$]
    We first note that by analyticity of $\LT{w_j}$ on $\Sigma$, we have
        $$ \int_{-\infty}^{\infty} x^{k-1}w_j(x) e^{-sx} dx = (\LT{w_j})^{(k-1)}(s),\quad s\in\Sigma. $$
    Since $r<n$, Proposition \ref{ADT_SR} then leads to a first order differential equation for $\LT{w_1}$ of the form
        $$ \frac{(\LT{w_1})'(s)}{\LT{w}_1(s)} = c \, \frac{\prod_{i=1}^{r} (s+a_i)^{d_1(i)}}{\prod_{i=1}^{r} (s+b_i)^{d_2(i)}},\quad s\in\Sigma,$$
    for some $a_i\in\C$, $b_i,c\in\R$ and $d_1,d_2:\{1,\dots,r\}\to\{0,1\}$ with all $d_1(i)=1$ or all $d_2(i)=1$. Its solution is given by
        $$ \LT{w}_1(s) = \LT{w}_1(s_0) \exp\left( c \int_{s_0}^{s} \frac{\prod_{i=1}^{r} (t+a_i)^{d_1(i)}}{\prod_{i=1}^{r} (t+b_i)^{d_2(i)}} dt \right),\quad s\in\Sigma,$$
    where $s_0\in\Sigma$. Therefore, by Proposition \ref{DT_ALG_COMP}, there exists an invertible upper-triangular matrix $U$ such that the functions $(w_{2,j})_{j=1}^r = (w_j)_{j=1}^r U$ satisfy
        $$ \LT{w_{2,j}}(s) = \exp\left( c \int_{s_0}^{s} \frac{\prod_{i=1}^{r} (t+a_i)^{d_1(i)}}{\prod_{i=1}^{r} (t+b_i)^{d_2(i)}} dt \right) \frac{s^{j-1}}{\prod_{i=1}^{j-1} (s+b_i)^{d_2(i)}},\quad s\in\Sigma.$$
    After the reparametrization $\vec{b}\mapsto (b_2,\dots,b_r,b_1)$ and modifying $d_2$ accordingly, we may write this as
    $$ \LT{w_{2,j}}(s) = \exp\left( c \int_{s_0}^{s} \frac{\prod_{i=1}^{r} (t+a_i)^{d_1(i)}}{\prod_{i=1}^{r} (t+b_i)^{d_2(i)}} dt \right) \frac{s^{j-1}}{\prod_{i=2}^{j} (s+b_i)^{d_2(i)}},\quad s\in\Sigma.$$
    It then remains to note that we can write
    $$\frac{\prod_{i=1}^{r} (t+a_i)^{d_1(i)}}{\prod_{i=1}^{r} (t+b_i)^{d_2(i)}} = \frac{\prod_{i=1}^{r} (t+a_i^\ast)^{d_1^\ast(i)}}{\prod_{i=1}^{r} (t+b_i)^{d_2(i)}} - \frac{1}{(t+b_1)^{d_2(1)}},$$
    by taking $a_i^\ast\in\C$ and $d_1^\ast:\{1,\dots,r\}\to\{0,1\}$ such that $ \prod_{i=1}^{r} (t+a_i^\ast)^{d_1^\ast(i)} = \prod_{i=1}^{r} (t+a_i)^{d_1(i)} + \prod_{i=2}^{r} (t+b_i)^{d_2(i)} $. In that case, all $d_1(i)=1$ is equivalent to all $d_1^\ast(i)=1$.
\end{proof}

In order to fully prove Theorem \ref{MOPE_ADT_CORE}, it remains to prove the equivalence between $ii)$ and $iii)$ for all $n\in\N$ (structurally, this is also what we need to prove Proposition \ref{MOPE_ADT_w_om}).

\begin{prop}
Let $n\in\N$. Suppose that the Laplace transforms of $w_{2,1},\dots,w_{2,r}\in L^{1}_{\LT{},\Sigma}(\R)$ are given by \eqref{MOPE_ADT_w} and the Laplace transform of $\om_{2,n}\in L^{1}_{\LT{},\Sigma}(\R)$ is given by \eqref{MOPE_ADT_om}. Then the functions in $\mathcal{W}(n;w_{2,1},\dots,w_{2,r})$ are of multiplicative derivative type w.r.t. $\om_{2,n}$.
\end{prop}

\begin{proof}
     We will show this by induction on $n\in\N$ via Proposition~\ref{ADT_SR}. For $n\in\{1,\dots,r\}$, by definition, we have 
     $$\frac{\LT{w_{2,n}}(s)}{\LT{w_{2,1}}(s)}=\frac{s^{n-1}}{\prod_{i=2}^{n}(s+b_i)^{d_2(i)}},\quad s\in\Sigma,$$
     and $ii)$ is satisfied. Hence, $\mathcal{W}(n;w_{2,1},\dots,w_{2,r})$ is of additive derivative type w.r.t. some $\om_n:\R_{>0}\to\C$ with
        $$ \LT{\om_n}(s)=\frac{\LT{w_{2,1}}(s)}{\prod_{i=2}^{n}(s+b_i)^{d_2(i)}},\quad s\in\Sigma, $$
    and thus $\LT{\om_n}=\LT{\om_{2,n}}$ on $\Sigma$ as desired. 
    For $n=r+1$, we take the logarithmic derivative of $\LT{w_{2,1}}(s)$ to obtain
        $$ \frac{(\LT{w_{2,1}})'(s)}{\LT{w_{2,1}}(s)} = c \prod_{i=1}^{r} \frac{(s+a_i)^{d_1(i)}}{(s+b_i)^{d_2(i)}} - \frac{d_2(1)}{s+b_1},\quad s\in\Sigma,. $$
    After writing the right hand side on a common denominator, it is straightforward to check that the degree condition in $ii)$ is satisfied and that the denominator is $\prod_{i=1}^r (s+b_i)^{d_2(i)}$. Hence, $\mathcal{W}(r+1;w_{2,1},\dots,w_{2,r})$ is of additive derivative type w.r.t. some $\om_{r+1}:\R_{>0}\to\C$ with
        $$ \LT{\om_{r+1}}(s)=\frac{\LT{w_{2,1}}(s)}{\prod_{i=1}^{n}(s+b_i)^{d_2(i)}},\quad s\in\Sigma, $$
    and thus $\LT{\om_{r+1}}=\LT{\om_{2,r+1}}$ on $\Sigma$ as desired. 
    Suppose that $\mathcal{W}(m;w_{2,1},\dots,w_{2,r})$ is of additive derivative type w.r.t. $\om_{2,m}$ for all $m<n$ and $n=r\eta+j\geq r+2$ with $\eta\in\Z_{\geq 1}$ and $j\in\{0,\dots,r-1\}$. Since $i)$ is satisfied for $m\in\{1,\dots,n-1\}$, we obtain $ii)$ for $m\in\{1,\dots,n-1\}$. We need to show that $ii)$ also holds for $m=n$ so that, according to $i)$, $\mathcal{W}(n;w_{2,1},\dots,w_{2,r})$ is of additive derivative type w.r.t. some $\om_n$. If $(\LT{w_{2,j}})^{(\eta)}(s)/\LT{w_{2,1}}(s)$ is of the appropriate form, we should then recover $\LT{\om_n}=\LT{\om_{2,n}}$ on $\Sigma$. We know that
        $$ \frac{(\LT{w_{2,j}})^{(\eta-1)}(s)}{\LT{w_{2,1}}(s)} = \frac{p_{\eta-1,j}(s)}{\prod_{i=1}^{r} (s+b_i)^{d_2(i)(\eta-1+1_{2\leq i\leq j})}},\quad s\in\Sigma, $$
    where the maximum of the degrees of the numerator and denominator is $r(\eta-1)+j-1$. By taking the derivative, we obtain
    $$\begin{aligned}
        \frac{(\LT{w_{2,j}})^{(\eta)}(s)}{\LT{w_{2,1}}(s)} = \frac{p_{\eta,j}(s)}{\prod_{i=1}^{r} (s+b_i)^{d_2(i)(\eta-1+1_{2\leq i\leq j})}},\quad s\in\Sigma,
    \end{aligned}$$
    where 
    $$p_{\eta,j}(s) =  p_{\eta-1,j}'(s) - p_{\eta-1,j}(s) \sum_{i=1}^r \frac{d_2(i) (\eta+1_{2\leq i\leq j})}{s+b_i} + p_{\eta-1,j}(s) \frac{(\LT{w_{2,1}})'(s)}{\LT{w_{2,1}}(s)},\quad s\in\Sigma.$$
    We know that 
    $$\frac{(\LT{w_{2,1}})'(s)}{\LT{w_{2,1}}(s)} = \frac{p_{1,1}(s)}{\prod_{i=1}^{r} (s+b_i)^{d_2(i)}},\quad s\in\Sigma,$$
     with $\deg p_{1,1} = r$ or all $d_2(i)=1$, so $(\LT{w_{2,j}})^{(\eta)}(s)/\LT{w_{2,1}}(s)$ is a rational function with denominator $\prod_{i=1}^{r} (s+b_i)^{d_2(i)(\eta+1_{2\leq i\leq j})} $. The degree of the numerator is at most $r\eta+j-1$. If the degree of the denominator is not $r\eta+j-1$, we must have that $d_2(i_0)=0$ for some $i_0$. In that case, $\deg p_{\eta-1,j}=r(\eta-1)+j-1$ and $\deg p_{1,1}= r$. Since the first two terms of $p_{\eta,j}(s)$ lead to a polynomial of degree at most $r\eta+j-2$ in the numerator, the degree of the numerator must be $\deg p_{\eta-1,j}+\deg p_{1,1}= r\eta+j-1$. Note that if $\vec{n}\in\mathcal{S}^r$ has $\sz{n}=n$, then $n_i = \eta+1$ for $i\in\{1,\dots,j\}$ and $n_i = \eta$ for $i\in\{j+1,\dots,r\}$, so that we can write the denominator as $\prod_{i=1}^r (s+b_i)^{d_2(i)n_i}/(s+b_1)^{d_2(1)}$. We can therefore conclude that $\mathcal{W}(n;w_{2,1},\dots,w_{2,r})$ is of additive derivative type w.r.t. a function $\om_n:\R\to\C$ with    
        $$ \LT{\om_n}(s) = \frac{(s+b_1)^{d_2(1)}\LT{w_{1,1}}(s)}{\prod_{i=1}^r (s+b_i)_{n_i}^{d_2(i)}},\quad s\in\Sigma. $$
    Together with the identity,
        $$ \LT{w_{1,1}}(s) = \exp\left( c  \int_{s_0}^{s} \prod_{i=1}^r \frac{(t+a_i)^{d_1(i)}}{(t+b_i)^{d_2(i)}} dt \right) \frac{1}{(s+b_1)^{d_2(1)}} ,\quad s\in\Sigma,$$
    this yields $\LT{\om_n}=\LT{\om_{2,n}}$ on $\Sigma$ as desired.
\end{proof}

\section*{Acknowledgements}

I would like to thank Andrei Martínez-Finkelshtein for introducing me to the notion of finite free convolutions, for explaining its connection to certain families of multiple orthogonal polynomials and for the many helpful discussions related to the open problem discussed in Section~\ref{FFP}. I'm also grateful to Mario Kieburg for providing some insight on the characterizations presented here and Arno Kuijlaars for suggesting several improvements in the presentation of the main results. Finally, I acknowledge PhD project 3E210613 funded by BOF KU Leuven and the grant KAW 2023.0216 from the Knut and Alice Wallenberg Foundation.

\end{document}